\documentclass[final]{siamart190516}

\usepackage{graphicx}
\usepackage{amsmath}
\usepackage{amsfonts}
\usepackage{tikz}
\usepackage{multirow}
\usepackage{mathtools}
\usepackage{arydshln}
\newsiamremark{remark}{Remark}
\usepackage{cleveref}
\usepackage[space,noadjust,nocompress]{cite}

\newcommand{\TheTitle}{Analysis of parallel Schwarz algorithms for time-harmonic problems using block Toeplitz matrices}
\newcommand{\TheAuthors}{N. Bootland, V. Dolean, A. Kyriakis, and J. Pestana}

\headers{Analysis of parallel Schwarz algorithms}{\TheAuthors}

\title{{\TheTitle}\thanks{The first two authors gratefully acknowledge support from the EPSRC grant EP/S004017/1. The fourth author gratefully acknowledges support from the EPSRC grant EP/R009821/1.}}%

\author{N. Bootland%
	\thanks{Department of Mathematics and Statistics, University of Strathclyde, Glasgow, UK (\email{niall.bootland@strath.ac.uk}, \email{alexandros.kyriakis@strath.ac.uk}, \email{jennifer.pestana@strath.ac.uk}).}%
	\and
	V. Dolean%
	\thanks{Department of Mathematics and Statistics, University of Strathclyde, Glasgow, UK and Laboratoire J.A.~Dieudonn\'e, CNRS, University C\^ote d'Azur, Nice, France (\email{work@victoritadolean.com}).}%
	\and
	A. Kyriakis\footnotemark[2]%
	\and
	J. Pestana\footnotemark[2]
}

\Crefname{ALC@unique}{Line}{Lines}

\ifpdf
\hypersetup{
	pdftitle={\TheTitle},
	pdfauthor={\TheAuthors},
	pdfkeywords={domain decomposition methods, Helmholtz equations, Maxwell equations, Schwarz algorithms, one-level methods, block Toeplitz matrices}
}
\fi

\begin{document}

\maketitle

\begin{abstract}
	In this work we study the convergence properties of the one-level parallel Schwarz method with Robin transmission conditions applied to the one-dimensional and two-dimensional Helmholtz and Maxwell's equations. One-level methods are not scalable in general. However, it has recently been proven that when impedance transmission conditions are used in the case of the algorithm applied to the equations with absorption, under certain assumptions, scalability can be achieved and no coarse space is required. We show here that this result is also true for the iterative version of the method at the continuous level for strip-wise decompositions into subdomains that can typically be encountered when solving wave-guide problems. The convergence proof relies on the particular block Toeplitz structure of the global iteration matrix. Although non-Hermitian, we prove that its limiting spectrum has a near identical form to that of a Hermitian matrix of the same structure. We illustrate our results with numerical experiments.
\end{abstract}



\section{Introduction}
\label{sec:intro}

Time-harmonic wave propagation problems, such as those arising in electromagnetic and seismic applications, are notoriously difficult to solve for several reasons. At the continuous level, the underlying boundary value problems lead to non self-adjoint operators (when impedance boundary conditions are used). The discretisation of these operators by a Galerkin method requires an increasing number of discretisation points as the wave number grows in order to avoid the pollution effect, that is a shift in the numerical wave velocity with respect to the continuous one \cite{Babuska:1997:IPE}. This leads to increasingly large linear systems with non-Hermitian matrices that are difficult to solve by classical iterative methods \cite{Ernst:12:NAM}.

In the past two decades, different classes of efficient solvers and preconditioners have been devised; see the review \cite{Gander:2018:SIREV} and references therein. One important class is based on domain decomposition methods \cite{Dolean:15:DDM}, which are a good compromise between direct and iterative methods. Some of these domain decomposition methods rely on improving the transmission conditions, that pass data between subdomains, to give optimised transmission conditions; see the seminal work on Helmholtz equations \cite{Gander:2007:OSM} and its extension to Maxwell's equations \cite{Dolean:2015:ETC, Dolean:09:OSM, Dolean:08:DDM, ElBouajaji:12:OSM} as well as to elastic waves \cite{Brunet:2019:NDD,Mattesi:2019:ABC}. For large-scale problems, in order to achieve robustness with respect to the number of subdomains (scalability) and the wave number, two-level domain decomposition solvers have been developed in recent years: they are based on the idea of using the absorptive counterpart of the equations as a preconditioner, which in turn is solved by a domain decomposition method. These methods were successfully applied to Helmholtz and Maxwell's equations, which arise naturally in different applications \cite{Bonazzoli:2019:ADD,Dolean:2020:IFD,Graham:2017:RRD}.

However, an alternative idea emerged in the last few years by observing that, when using Robin or impedance transmission conditions, under certain assumptions involving the physical and numerical parameters of the problem (i.e., absorption, size of the subdomains, etc.) one-level Schwarz algorithms can scale weakly (have a convergence rate that does not deteriorate as the number of subdomains grows) without the addition of a second level \cite{Gong:2021:DDP,Graham:2020:DDI}. The notion of scalability here applies over a family of problems rather than for a fixed problem. In essence, weak scalability is achieved such that the convergence rate of the domain decomposition method does not deteriorate for harder problems in the family when an appropriate number of subdomains is used. In other words, adding more subdomains allows us to solve harder problems while achieving the same convergence rate.

Achieving scalability without a coarse space in the case of a decomposition into chains of subdomains was first observed for problems arising in computational chemistry; see \cite{Cances:2013:DDI}. However, the first true scalability analysis, based on Fourier techniques, was developed in \cite{Ciaramella:2017:APS} for a classical parallel Schwarz method on a rectangular chain of fixed-size subdomains and provides the first concrete construction of the Schwarz iteration operator in Fourier space. This technique was extended in \cite{Chaouqui:2018:OSC} to other types of one-level methods. Weak scalability results for the Laplace problem have been proven for more general chain-type geometries using various techniques, such as the maximum principle in \cite{Ciaramella:2018:APSa} and a fully variational analysis in \cite{Ciaramella:2018:APSb}. The most recent work on the topic without restrictive assumptions can be found in \cite{Ciaramella:2020:OSS} where a propagation-tracking analysis based on graph theory and the maximum principle permitted a scalability analysis for very general decompositions. To our knowledge, there is no such analysis on Schwarz methods for time-harmonic wave propagation problems, where previous techniques no longer extend as the nature of the underlying equations is very different.

In our work, we would like to explore this idea of weak scalability at the continuous level (independent of the discretisation) for a strip-wise decomposition into subdomains as it arises naturally in the solution of wave-guide problems. While in \cite{Gong:2021:DDP,Graham:2020:DDI} the family of problems is parametrised by the wave number $k$ and the focus is on $k$-robustness, here we focus on the weak scalability aspect for a family consisting of a growing chain of fixed-size subdomains. Nonetheless, we will see that $k$-robustness in certain scenarios can easily be derived from our theory. The main contributions of the paper are the following:
\begin{itemize}
	\item We provide analysis of the limiting spectrum, as the number of subdomains grows, for a one-level Schwarz method applied to a strip-wise decomposition. While our analysis is limited to this simple yet realistic configuration (wave propagation in a rectangular wave-guide with Dirichlet conditions on the top and bottom boundaries and Robin condition at its ends), it is valid at the continuous level both for one-dimensional and two-dimensional Helmholtz and Maxwell's equations.
	\item We build on the formalism of iteration matrices acting on interface data introduced in \cite{Chaouqui:2018:OSC} (where Schwarz methods using strip-wise decompositions were analysed for Laplace's equation), but here we are able to characterise the entire spectrum of these iteration matrices by using their block Toeplitz structure, even if upper bounds on the iteration matrix norm could have been derived in a similar manner.
	\item Despite the fact that the block Toeplitz structure is non-Hermitian, and thus results from the standard literature on Toeplitz matrices do not apply in a straightforward manner, we prove that the limiting spectrum of the iteration matrices as their size grows (corresponding to an increasing number of subdomains) tends to the limit predicted by the eigenvalues of the symbol of the block Toeplitz matrix, except perhaps for two additional eigenvalues. This novel approach, utilising the limiting spectrum, is quite general and can be applied to other problems as an analysis tool for domain decomposition methods where such block Toeplitz structure arises naturally.
	\item We show that the limiting spectrum is descriptive of what is observed in practice numerically, even for a relatively small number of subdomains.
	\item As a corollary to our theory we show that, in certain scenarios and with $k$-dependent domain decomposition parameters, the one-level method can be $k$-robust as the wave number $k$ increases; in the Maxwell case we believe this to be a novel result.
\end{itemize}
The structure of the paper is as follows: In \Cref{sec:toeplitz} we present our results on the limiting spectrum of a non-Hermitian block Toeplitz matrix whose characteristic polynomial verifies a three-term recurrence. In \Cref{sec:1dcase,sec:2dcase} we apply these results to the analysis of the iterative Schwarz algorithm in the one-dimensional and two-dimensional cases. We illustrate the theory with numerical results in \Cref{sec:num}. Finally, \Cref{sec:conc} draws together our conclusions.

The codes used to provide numerical results in this work as well as several \texttt{Maple} worksheets, that confirm some of the more involved calculations required in \Cref{sec:1dcase,sec:2dcase}, are provided at \url{https://github.com/vicdolean/schwarz}.

\section{A non-Hermitian block Toeplitz structure}
\label{sec:toeplitz}

Consider a non-Hermitian block Toeplitz matrix $\mathcal{T} \in \mathbb{C}^{2m\times2m}$ of the form
\begin{subequations}
	\label{NonHermitianBlockToeplitzStructure}
	\begin{align}
		\label{NonHermitianBlockToeplitzStructure-Matrix}
		\mathcal{T}
		= \left[\begin{array}{ccccc}
			A_{0} & A_{1} & & & \\
			A_{-1} & A_{0} & A_{1} & & \\
			& \ddots & \ddots & \ddots & \\
			& & A_{-1} & A_0 & A_1 \\
			& & & A_{-1} & A_0
		\end{array}\right],
	\end{align}
	where
	\begin{align}
		\label{NonHermitianBlockToeplitzStructure-Blocks}
		A_0 = \left[\begin{array}{cc} 0 & b \\ b & 0 \end{array}\right], \ A_1 = \left[\begin{array}{cc} a & 0 \\ 0 & 0 \end{array}\right], \ A_{-1} = \left[\begin{array}{cc} 0 & 0 \\ 0 & a \end{array}\right],
	\end{align}
\end{subequations}
for some non-zero complex coefficients $a$ and $b$. We will see in the sections that follow that such non-Hermitian block Toeplitz structures arise naturally for iterative Schwarz algorithms applied to wave propagation problems. We are interested in a characterisation of the complete spectrum of the matrix $\mathcal{T}$ in \eqref{NonHermitianBlockToeplitzStructure} when its dimension becomes large. This will equate to the number of subdomains $N$ in the Schwarz method being large. The coefficients $a$ and $b$ stem from the particular PDE and domain decomposition used; we consider them to be fixed independent of the dimension of $\mathcal{T}$, and thus $N$, which corresponds to fixed-size subdomains.

The so-called Szeg\H{o} formula enables the asymptotic spectrum, i.e., the spectrum as $m \to \infty$, of a wide class of Hermitian block Toeplitz matrices to be characterised by the eigenvalues of an associated matrix-valued function called the (block) symbol~\cite{Tillinota}. For non-Hermitian matrices, analogous results do not exist in general \cite{Tillinota}, but do hold when the union of the essential ranges of the eigenvalues of the block symbol has empty interior and does not disconnect the complex plane \cite{DNS12}. Unfortunately, $\mathcal{T}$ in \eqref{NonHermitianBlockToeplitzStructure} has symbol $F(z) = A_{-1}z+ A_0 + A_{1}z^{-1}$ and, for relevant values of $a$ and $b$, the union of essential ranges is a closed curve. Additional characterisations of the asymptotic spectrum of (block) banded Toeplitz matrices are available~\cite{Hirs67,ScSp60,Wido74}, but these do not provide explicit formulae for the eigenvalues, as we shall in \Cref{theorem:LimitingSpectrum}. Other formulae for the eigenvalues~\cite{SaMo13} and determinant~\cite{Time87} of block tridiagonal Toeplitz matrices are known, however, they are applicable only when $A_1$ (or $A_{-1}$) is nonsingular.

We also remark that the matrix $\mathcal{T}$ will be an iteration matrix in the Schwarz algorithms we later consider. Hence, to prove convergence of these Schwarz methods it would be sufficient to bound the spectral radius of $\mathcal{T}$, for example using a matrix norm. It is straightforward to see that $\|\mathcal{T}\|_\infty = |a| + |b|$, and it is also possible to show, using \cite[Corollary 3.5]{Serr02}, that
\begin{align*}
	\|\mathcal{T}\|_2 \le \max \left\lbrace \sqrt{|a|^2 \pm 2\Re(a \overline{b}) + |b|^2} \right\rbrace.
\end{align*}
However, since $a$ and $b$ are complex neither norm is straightforward to bound above by 1. Additionally, characterising the full spectrum provides more information than the spectral radius alone. Accordingly, in this section we derive the limiting spectrum of $\mathcal{T}$.

In order to establish a result on the spectrum of $\mathcal{T}$, we first show that the characteristic polynomials of \eqref{NonHermitianBlockToeplitzStructure} for increasing $m$ obey a three-term recurrence relation.
\begin{lemma}[Three-term recurrence and generating function]
	\label{lemma:ThreeTermRecurrence}
	Let $p_{m}(z)$ denote the characteristic polynomial of the block Toeplitz matrix $\mathcal{T} \in \mathbb{C}^{2m\times2m}$ defined in \eqref{NonHermitianBlockToeplitzStructure}. Then $p_{m}(z)$ satisfies the three-term recurrence relation
	\begin{align}
		\label{ThreeRecurrenceRelation-General}
		p_{m}(z) + B(z) p_{m-1}(z) + A(z) p_{m-2}(z) & = 0, & \text{for} \ m \ge 2,
	\end{align}
	with $A(z) = a^2 z^2$ and $B(z) = - z^2 + b^2 - a^2$ and where $p_{0}(z) = 1$ and $p_{1}(z) = z^{2} - b^{2}$. Furthermore, this recurrence relation is encoded in the generating function
	\begin{align}
		\label{GeneratingFunctionDefinition}
		\sum_{m=0}^{\infty} p_{m}(z) t^{m} = \frac{N(t,z)}{D(t,z)},
	\end{align}
	where
	\begin{subequations}
		\label{GeneratingFunctionDAndN}
		\begin{align}
			\label{GeneratingFunctionD}
			D(t,z) & = 1 + B(z)t + A(z)t^2,\\
			\label{GeneratingFunctionN}
			N(t,z) & = p_{0}(z) + (p_{1}(z) + B(z)p_{0}(z))t.
		\end{align}
	\end{subequations}
	Thus, in our case, $D(t,z) = 1 - (z^2-b^2+a^2)t + a^2 z^2 t^2$ while $N(t,z) = 1 - a^2 t$.
\end{lemma}
\begin{proof}
	We first prove the recurrence relation. Let $D_{m}$ be the $2m\times2m$ matrix whose determinant is the characteristic polynomial of $\mathcal{T}$ in the variable $z$. Note that the first two characteristic polynomials are
	\begin{subequations}
		\label{FirstTwoCharacteristicPolynomials}
		\begin{align}
			p_{1}(z) &= \det(D_{1}) = \left|\begin{array}{cc} -z & b \\ b & -z \end{array}\right| = z^{2} - b^{2}, \\
			p_{2}(z) &= \det(D_{2}) = \left|\begin{array}{cccc} -z & b & a & 0 \\ b & -z & 0 & 0 \\ 0 & 0 & -z & b \\ 0 & a & b & -z \end{array}\right| = (z^{2} - b^{2})^{2} - a^{2} b^{2}.
		\end{align}
	\end{subequations}
	To derive a recurrence relation, let us also define the intermediary determinants $r_{m}(z)$ which arise as the minor of $D_{m}$ having removed the second row and first column,
	\begin{align*}
		r_{m}(z) \vcentcolon= \left|\begin{array}{c;{2pt/2pt}cccc} b & a & 0 & 0 & \cdots \\ \hdashline[2pt/2pt] 0 & \multicolumn{4}{c}{\multirow{4}{*}{$D_{m-1}$}} \\ a & & & & \\ 0 & & & & \\ \vdots & & & & \end{array}\right| = \left|\begin{array}{c;{2pt/2pt}cc;{2pt/2pt}cc} b & a & 0 & 0 & \cdots \\ \hdashline[2pt/2pt] 0 & -z & b & a & 0 \\ a & b & -z & 0 & 0 \\ \hdashline[2pt/2pt] 0 & 0 & 0 & \multicolumn{2}{c}{\multirow{2}{*}{$D_{m-2}$}} \\ \vdots & 0 & a & & \end{array}\right| = b \, p_{m-1}(z) + a^{2} \, r_{m-1}(z),
	\end{align*}
	where we use the cofactor expansion of the determinant. Similarly, for $p_{m}(z)$ we obtain
	\begin{align*}
		p_{m}(z) = z^{2} \, p_{m-1}(z) - b \, r_{m}(z) = (z^{2} - b^{2}) \, p_{m-1}(z) - a^{2} b \, r_{m-1}(z).
	\end{align*}
	We can then rearrange this relation to give an expression for $r_{m-1}(z)$ in terms of $p_{m}(z)$ and $p_{m-1}(z)$. Substituting this into the recurrence for $r_{m}(z)$ above, along with the equivalent expression for $r_{m}(z)$, yields the desired recurrence relation
	\begin{align}
		\label{ThreeTermRecurrence}
		p_{m+1}(z) & = (z^{2} - b^{2} + a^{2}) \, p_{m}(z) - a^{2}z^{2} \, p_{m-1}(z),
	\end{align}
	where $A(z) \vcentcolon= a^2 z^2$ and $B(z) \vcentcolon= - z^2 + b^2 - a^2$. Finally, note that setting $p_{0} = 1$ is consistent with this recurrence relation and initial characteristic polynomials \eqref{FirstTwoCharacteristicPolynomials}.

	To show the equivalence of the generating function, we multiply \eqref{ThreeRecurrenceRelation-General} by $t^{m}$ and sum over $m\ge2$ before adding relevant terms to isolate $\sum_{m=0}^{\infty} p_{m}(z)t^{m}$ as follows
	\begin{align*}
		& \sum_{m=2}^{\infty} \left[p_{m}(z) + B(z) p_{m-1}(z) + A(z) p_{m-2}(z)\right]t^{m} = 0 \\
		\iff & \sum_{m=0}^{\infty} \left[1 + B(z) t + A(z) t^{2}\right]p_{m}(z)t^{m} = p_{0}(z) + \left(p_{1}(z)+B(z)p_{0}(z)\right)t \\
		\iff & \sum_{m=0}^{\infty} p_{m}(z)t^{m} = \frac{p_{0}(z) + \left(p_{1}(z)+B(z)p_{0}(z)\right)t}{1 + B(z) t + A(z) t^{2}}.
	\end{align*}
	Substituting in the appropriate values gives $D(t,z) = 1 - (z^2-b^2+a^2)t + a^2 z^2 t^2$ and $N(t,z) = 1 - a^2 t$ in our case, as required.
\end{proof}

Before continuing, we remark on the convergence of the Maclaurin series in $t$ of the generating function. Note that the Maclaurin series of any rational function (without a pole at 0) satisfies a linear recurrence relation, which can be seen by following backwards an analogous argument to that in the above proof. Moreover, the Maclaurin series is convergent (to the rational function) on the open disc centred at 0 with a radius equal to the minimum root of the denominator in absolute value; this can be discerned from a partial fractions decomposition (over $\mathbb{C}$) and noting that it is a (finite) sum of geometric series. As such, in our present case, $p_m(z)$ are precisely the coefficients in the Maclaurin series for any given $z$ since the denominator is such that 0 is never a pole of the generating function and so there is always a non-trivial disc where the series converges.

We now introduce a useful tool that will help us to characterise the spectrum of \eqref{NonHermitianBlockToeplitzStructure}: the $q$-analogue of the discriminant known as the $q$-discriminant \cite{Tran:2014:CBD}. The $q$-discriminant of a polynomial $P_{n}(t)$ of degree $n$ with leading coefficient $p$ is defined as
\begin{align}
	\label{q-discriminant}
	\mathrm{Disc}_{t}(P_{n};q) = p^{2n-2} q^{n(n-1)/2} \prod_{1 \le i < j \le n} (q^{-1/2}t_{i} - q^{1/2}t_{j})(q^{1/2}t_{i} - q^{-1/2}t_{j}),
\end{align}
where $t_{i}$, $1 \le i \le n$, are the roots of $P_{n}(t)$. A key point is that the $q$-discriminant is zero if and only if a quotient of roots $t_{i}/t_{j}$ equals $q$. Note that as $q\rightarrow1$ the $q$-discriminant becomes the standard discriminant of a polynomial.

In particular, we will consider the $q$-discriminant of the denominator $D(t,z)$ as a quadratic in $t$. Direct calculation using the quadratic formula yields
\begin{align}
	\label{q-discriminant-of-D}
	\mathrm{Disc}_{t}(D(t,z);q) = q \left(B(z)^{2} - (q + q^{-1} + 2)A(z)\right),
\end{align}
for any $q\neq0$. If $q$ is a quotient of the two roots in $t$ of $D(t,z)$ then \eqref{q-discriminant-of-D} is zero and so $q$ must satisfy
\begin{align}
	\label{DefiningCurve}
	\frac{B(z)^{2}}{A(z)} = q + q^{-1} + 2,
\end{align}
where, in general, $q$ will depend on $z$. The $q$-discriminant condition \eqref{DefiningCurve} for $D(t,z)$ will be crucial in what follows since it will allow us to characterise roots of $p_m(z)$ in terms of the quotient $q$. We now state our main result on the limiting spectrum of $\mathcal{T}$ as its dimension becomes large in which we adapt some ideas from \cite{Tran:2014:CBD} for finding roots of polynomials verifying a three-term recurrence but now with a different generating function.

\begin{theorem}[Limiting spectrum]
	\label{theorem:LimitingSpectrum}
	The limiting spectrum, as $m\rightarrow\infty$, of the block Toeplitz matrix $\mathcal{T} \in \mathbb{C}^{2m\times2m}$, defined in \eqref{NonHermitianBlockToeplitzStructure}, lies on the curve defined by
	\begin{align}
		\label{LimitingCurve}
		\lambda_{\pm}(\theta) &= a \cos(\theta) \pm \sqrt{b^2 - a^2 \sin^2(\theta)}, & & \theta \in [-\pi,\pi],
	\end{align}
	except perhaps for the eigenvalues
	\begin{align}
		\label{LimitingPoints}
		\lambda = \pm \sqrt{\tfrac{1}{2}b^{2} - a^{2}},
	\end{align}
	which can only occur if $|a^{2}|>|\frac{1}{2}b^{2} - a^{2}|$.
\end{theorem}
\begin{proof}
	Suppose that $z_{m}$ is a root of the characteristic polynomial $p_{m}(z)$ for $m\ge2$. If $z_{m}=0$ then we must have that $a^{2} = b^{2}$. To see this, assume for a contradiction that $a^{2} \neq b^{2}$, then $B(0)\neq0$ while $A(0)=0$ and $p_{m}(0)=0$ and thus the recurrence relation \eqref{ThreeRecurrenceRelation-General} gives that $p_{m-1}(0)=0$. Following this recursion down to $m=2$ gives that $p_{1}(0)=0$, which is false as $b\neq0$. Further, if $p_{m}(0) = 0$ then also $p_{m+1}(0) = 0$ by \eqref{ThreeRecurrenceRelation-General} since $A(0)=0$ and so a sequence of zero roots occurs as $m$ increases giving $0$ in the limiting spectrum. This case is covered by choosing $\theta = \frac{\pi}{2}$ in \eqref{LimitingCurve} and noting that $a^{2} = b^{2}$ must hold. As such, for the remainder of the proof we assume that $z_{m}\neq0$.

	Now consider the denominator $D(t,z_{m})$. Since $A(z_{m})\neq0$ by the assumption that $z_{m}\neq0$, the denominator as a quadratic in $t$ has two roots $t_{1}$ and $t_{2}$. Note that, by Vieta's formula for the product of roots, neither of these two roots can be zero since $t_{1}t_{2}A(z_{m}) = 1$. If $t_{1}=t_{2}$ then the (standard) discriminant of $D(t,z_{m})$ is zero, giving $B(z_{m})^2-4A(z_{m})=0$. Solving for $z_{m}$ given our expressions for $A(z)$ and $B(z)$ yields solutions $z_{m} = \pm(a \pm b)$ for all choices of signs. These cases are also covered by \eqref{LimitingCurve} when $\theta=0$ or $\theta=\pi$.

	As such, we now assume that $t_{1} \neq t_{2}$ and so $D(t,z_{m})=A(z_{m})(t-t_{1})(t-t_{2})$. Considering the generating function \eqref{GeneratingFunctionDefinition} we observe that
	\begin{align}
		\nonumber
		\frac{N(t,z_{m})}{D(t,z_{m})} & = \frac{1-a^{2}t}{A(z_{m})(t-t_{1})(t-t_{2})} = \frac{1-a^{2}t}{A(z_{m})(t_{1}-t_{2})}\left(\frac{1}{t-t_{1}} - \frac{1}{t-t_{2}}\right) \\
		\nonumber
		& = \frac{1-a^{2}t}{A(z_{m})(t_{1}-t_{2})} \sum_{m=0}^{\infty} \frac{t_{1}^{m+1}-t_{2}^{m+1}}{t_{1}^{m+1}t_{2}^{m+1}} t^{m} \\
		\label{ExpansionOfGeneratingFunction}
		& = \frac{1}{A(z_{m})(t_{1}-t_{2})} \sum_{m=1}^{\infty} \left[ \frac{t_{1}^{m+1}-t_{2}^{m+1}}{t_{1}^{m+1}t_{2}^{m+1}} - a^{2} \frac{t_{1}^{m}-t_{2}^{m}}{t_{1}^{m}t_{2}^{m}} \right] t^{m} + 1.
	\end{align}
	The sum introduced in the second line is the Maclaurin series in $t$ and, as the difference of two geometric series, is convergent in the open disc $|t| < \min\{|t_{1}|,|t_{2}|\}$. Note that this is non-trivial since neither $t_{1}$ or $t_{2}$ are zero. In \eqref{ExpansionOfGeneratingFunction} we identify that the coefficient of $t^{m}$ is exactly $p_{m}(z_{m})$. Thus, as $z_{m}$ is a root of $p_{m}(z)$, the coefficient of $t^{m}$ in \eqref{ExpansionOfGeneratingFunction} must be zero. Now suppose $t_{1}=qt_{2}$ for some quotient $q\neq0$ (as neither $t_{1}$ nor $t_{2}$ is zero), then this condition on the coefficient of $t^{m}$ translates into
	\begin{align*}
		\frac{q^{m+1}-1}{q^{m+1}t_{2}^{m+1}} - a^{2} \frac{q^{m}-1}{q^{m}t_{2}^{m}} = 0 \implies q^{m+1} - 1 = a^{2}t_{2}q(q^{m}-1).
	\end{align*}
	Since $t_{1}t_{2}A(z_{m}) = 1$, we deduce that $t_{2} = \pm (A(z_{m})q)^{-1/2}$ and thus $q$ must solve
	\begin{align*}
		\left(q^{m+1} - 1\right)^{2} = \frac{a^4}{A(z_{m})} \left(q^{m} - 1\right)^{2}q.
	\end{align*}
	Let us define the coefficient, depending on $z_{m}$,
	\begin{align}
		\label{q-PolynomialCoeffient}
		c_{m} = \frac{a^4}{A(z_{m})} = \frac{a^2}{z_{m}^2}.
	\end{align}
	Then $q$ must be a root of the $2m+2$ degree polynomial
	\begin{align}
		\label{q-Polynomial}
		f_{m}(q) = q^{2m+2} - c_{m}q^{2m+1} + 2(c_{m}-1)q^{m+1} - c_{m}q + 1.
	\end{align}
	In order to characterise the roots of \eqref{q-Polynomial} we will make use of the following corollary of Rouch\'e's theorem (see, e.g., \cite[Section 5.3.2]{Krantz99}): for a polynomial $f$ of degree $d$ with coefficients $\left\lbrace\alpha_{j}\right\rbrace_{j=0}^{d}$, if $R>0$ is such that for an integer $0 \le k \le d$ we have
	\begin{align}
		\label{RoucheCondition}
		|\alpha_{0}| + \ldots + |\alpha_{k-1}|R^{k-1} + |\alpha_{k+1}|R^{k+1} + \ldots + |\alpha_{d}|R^{d} < |\alpha_{k}|R^{k},
	\end{align}
	then there are exactly $k$ roots of $f$, counted with multiplicity, having absolute value less than $R$. In particular, we will use this result for the polynomial $f_{m}(q)$ with $k=0$, $k = 2m+1$ or $k=2m+2$.

	We first point out some facts about \eqref{q-Polynomial}. Note that $q=0$ is not a root of $f_{m}$. Moreover, by symmetry of the coefficients, we have (for $q \neq 0$)
	\begin{align}
		\label{q^-1-Polynomial}
		f_m(q^{-1}) = q^{-(2m+2)} f_{m}(q).
	\end{align}
	Thus, if $q_{m}$ is a root of $f_{m}$ then $q_{m}^{-1}$ is also a root. Further, since $f_m$ has a unique factorisation in $\mathbb{C}$, applying this both in the variable $q^{-1}$ and $q$ in \eqref{q^-1-Polynomial} shows that the multiplicities of the roots $q_{m}$ and $q_{m}^{-1}$ must be identical. This means that we only need to study roots with $|q_{m}|\le1$, with roots outside the unit disc being precisely the reciprocal values of those inside the unit disc, or vice versa.

	We will use \eqref{RoucheCondition} to determine how many roots of $f_{m}(q)$ in \eqref{q-Polynomial} do not approach the unit circle as $m \to \infty$. This information, along with \eqref{DefiningCurve}, will allow us to determine conditions for $z_m$. A significant challenge is that the coefficient $c_m$ depends on $m$ and so we will need to consider several cases. To proceed, we let $\varepsilon>0$ be small. We will show that for all $m \ge M$, for a suitable $M(\varepsilon)$, all but potentially two roots of $f_{m}(q)$ lie in an annulus which shrinks to the unit circle as $\varepsilon \rightarrow 0$. The remaining two roots can only persist if $|c_m| > 1$ and, should they exist, consist of a root $s_m$ close to $c_m^{-1}$ and the corresponding reciprocal root outside the unit circle. Given $\varepsilon>0$, for $m \ge M$ we consider three cases depending on $c_m$:
	\begin{enumerate}
		\item $|c_{m}| \le 1$,
		\label{Case1}
		\item $1 \le |c_{m}| \le (1+\varepsilon)^2$,
		\label{Case2}
		\item $|c_{m}| \ge (1+\varepsilon)^2$.
		\label{Case3}
	\end{enumerate}

	\paragraph{Case 1} To start the analysis we suppose that we are in case~\ref{Case1} so that $|c_{m}| \le 1$ and define $R_{-} = 1-\varepsilon$. Let $M_{1}$ be such that
	\begin{align*}
		4 R_{-}^{m+1} + R_{-}^{2m+1} + R_{-}^{2m+2} < \varepsilon
	\end{align*}
	for all $m \ge M_{1}$. Such an $M_{1}$ exists since $|R_{-}| < 1$. Then, for $m \ge M_{1}$, we have that
	\begin{align*}
		|c_{m}|R_{-} + 2|c_{m}-1|R_{-}^{m+1} + |c_{m}|R_{-}^{2m+1} + R_{-}^{2m+2} & \le R_{-} + 4 R_{-}^{m+1} + R_{-}^{2m+1}+ R_{-}^{2m+2} \\
		& < 1.
	\end{align*}
	Thus, for large enough $m$, by using $k=0$ in the corollary of Rouch\'e's theorem we deduce that there are no roots of $f_{m}$ with modulus less than $R_{-} = 1-\varepsilon$. In this case, by the reciprocal nature of the roots, for $m \ge M_{1}$ we conclude that all $2m+2$ roots $q_m$ of $f_m$ lie in the annulus
	\begin{align}
		\label{annulus-case-1}
		1 - \varepsilon \le |q_m| \le \frac{1}{1-\varepsilon}.
	\end{align}

	\paragraph{Case 2} We now turn to the analysis of case~\ref{Case2} where $1 \le |c_{m}| \le (1+\varepsilon)^2$. To aid in the next case we first relax this condition to consider $|c_{m}| \ge 1$ and prove a useful bound for all roots of $f_m$. Define $R_{+} = 1 + \varepsilon$ and let $M_2$ be such that
	\begin{align*}
		R_{+}^{-(2m+1)} + R_{+}^{-2m} + 4 R_{+}^{-m}  < \frac{\varepsilon}{1+\varepsilon}
	\end{align*}
	for all $m \ge M_{2}$. Now let $R_{\top} = |c_m| (1+\varepsilon) = |c_m| R_{+}$. We will want to show that
	\begin{align}
		\label{Rtopbound}
		1 + |c_{m}|R_{\top} + 2|c_{m}-1|R_{\top}^{m+1} + |c_{m}|R_{\top}^{2m+1} < R_{\top}^{2m+2},
	\end{align}
	in order to apply the corollary of Rouch\'e's theorem with $k = 2m+2$. To do so we consider dividing by $R_{\top}^{2m+2}$, in which case, for $m \ge M_{2}$, we have
	\begin{align*}
		& R_{\top}^{-(2m+2)} + |c_{m}|R_{\top}^{-(2m+1)} + 2|c_{m}-1|R_{\top}^{-(m+1)} + |c_{m}|R_{\top}^{-1} \\
		& = |c_{m}|^{-(2m+2)} R_{+}^{-(2m+2)} + |c_{m}|^{-2m} R_{+}^{-(2m+1)} + \frac{2|c_{m}-1|}{|c_m|} |c_{m}|^{-m} R_{+}^{-(m+1)} + R_{+}^{-1} \\
		& \le R_{+}^{-(2m+2)} + R_{+}^{-(2m+1)} + 4 R_{+}^{-(m+1)} + R_{+}^{-1} \\
		& < \frac{\varepsilon}{(1+\varepsilon)^2} + \frac{1}{1+\varepsilon} < 1.
	\end{align*}
	Thus we have the required inequality and deduce from the corollary of Rouch\'e's theorem that all $2m+2$ roots $q_m$ lie in the disc given by $|q_m| < |c_m| (1+\varepsilon)$. This will prove useful later in case~\ref{Case3}. For now we turn back to case~\ref{Case2} where $1 \le |c_{m}| \le (1+\varepsilon)^2$. Using this upper bound on $|c_{m}|$ and the reciprocal nature of the roots, we conclude that, for $m \ge M_{2}$, all $2m+2$ roots $q_m$ of $f_m$ lie in the annulus
	\begin{align}
		\label{annulus-case-2}
		\frac{1}{(1 + \varepsilon)^3} < |q_m| < (1+\varepsilon)^3.
	\end{align}

	\paragraph{Case 3} Finally, consider case~\ref{Case3} where $|c_{m}| \ge (1+\varepsilon)^2$. Let $R_{+} = 1 + \varepsilon$ and $M_2$ be as defined in case~\ref{Case2}. We will want to show that
	\begin{align}
		\label{Rplusbound}
		1 + |c_{m}|R_{+} + 2|c_{m}-1|R_{+}^{m+1} + R_{+}^{2m+2} < |c_{m}|R_{+}^{2m+1},
	\end{align}
	in order to apply the corollary of Rouch\'e's theorem with $k = 2m+1$. To do so we consider dividing by $|c_{m}|R_{\top}^{2m+1}$, in which case, for $m \ge M_{2}$, we have
	\begin{align*}
		& |c_{m}|^{-1} R_{+}^{-(2m+1)} + R_{+}^{-2m} + \frac{2|c_{m}-1|}{|c_m|} R_{+}^{-m} + |c_{m}|^{-1} R_{+} \\
		& \le R_{+}^{-(2m+1)} + R_{+}^{-2m} + 4 R_{+}^{-m} + (1+\varepsilon)^{-2} R_{+} \\
		& < \frac{\varepsilon}{1+\varepsilon} + \frac{1}{1+\varepsilon} = 1.
	\end{align*}
	Thus we have the required inequality and deduce from the corollary of Rouch\'e's theorem that $2m+1$ roots $q_m$ lie in the disc given by $|q_m| < 1+\varepsilon$. In this case, by the reciprocal nature of the roots, for $m \ge M_{2}$ we conclude that $2m$ roots $q_m$ of $f_m$ lie in the annulus
	\begin{align}
		\label{annulus-case-3}
		\frac{1}{1+\varepsilon} < |q_m| < 1 + \varepsilon.
	\end{align}

	We pause to note at this stage that, combining all three cases, we have just shown that all but potentially two roots of $f_m$ lie in a small annulus around the unit circle for $m \ge M = \max \lbrace M_{1}, M_{2} \rbrace$, independently of the value of $c_m$. In particular, this will be the largest annulus of the three cases which, for small $\varepsilon>0$, is that in \eqref{annulus-case-2}. Letting $\varepsilon\rightarrow0$ we deduce that all but potentially two roots of $f_m$ must tend to the unit circle as $m\rightarrow\infty$.

	The remaining question is what happens to the other two roots, which only appear in case~\ref{Case3}. We know from the bound in \eqref{Rtopbound} that, for large enough $m$, all roots satisfy $|q_m| < |c_m| (1+\varepsilon)$ while all but one satisfy $|q_m| < (1+\varepsilon)$. We now show that the remaining root in case~\ref{Case3} satisfies $|q_m| \ge |c_m| (1-\varepsilon)$ for large enough $m$. To do so let $R_{\bot} = |c_m|(1-\varepsilon) = |c_m| R_{-}$ and note that, assuming $\varepsilon$ is small enough ($\varepsilon<\frac{1}{3}$ suffices), then $R_{\bot} > 1$ since
	\begin{align*}
		R_{\bot} = |c_m|(1-\varepsilon) \ge (1+\varepsilon)^2 (1-\varepsilon) \ge 1 + \frac{\varepsilon}{2}.
	\end{align*}
	Now let $M_{3}$ be such that
	\begin{align*}
		\left(1+\frac{\varepsilon}{2}\right)^{-(2m+1)} + \left(1+\frac{\varepsilon}{2}\right)^{-2m} + 4 \left(1+\frac{\varepsilon}{2}\right)^{-m}  < \varepsilon
	\end{align*}
	for all $m \ge M_{3}$. We will want to show an identical bound to \eqref{Rplusbound} holds but now for $R_{\bot}$ in order to again use the corollary of Rouch\'e's theorem with $k = 2m+1$. We proceed in a similar manner and consider dividing by $|c_m|R_{\bot}^{2m+1}$, so that for $m \ge M_{3}$ we have
	\begin{align*}
		& |c_{m}|^{-1} R_{\bot}^{-(2m+1)} + R_{\bot}^{-2m} + \frac{2|c_{m}-1|}{|c_m|} R_{\bot}^{-m} + |c_{m}|^{-1} R_{\bot} \\
		& \le R_{\bot}^{-(2m+1)} + R_{\bot}^{-2m} + 4 R_{\bot}^{-m} + R_{-} \\
		& \le \left(1+\frac{\varepsilon}{2}\right)^{-(2m+1)} + \left(1+\frac{\varepsilon}{2}\right)^{-2m} + 4 \left(1+\frac{\varepsilon}{2}\right)^{-m} +  R_{-} \\
		& < 1.
	\end{align*}
	Thus we have the required inequality and deduce from the corollary of Rouch\'e's theorem that $2m+1$ roots $q_m$ lie in the disc given by $|q_m| < |c_m|(1-\varepsilon)$. Thus, for large enough $m$, we conclude that the single remaining root lies in the annulus $|c_m|(1-\varepsilon) \le |q_m| < |c_m|(1+\varepsilon)$.

	This result makes it clear that roots which do not tend to the unit circle persist only when we have $|c_m|$ values which stay bounded away from $1$ as $m\rightarrow\infty$, and their size is dictated by $c_m$. That is, for such roots to persist there must exist an infinite subsequence with $|c_m| > c > 1$ for some fixed $c$ and so we now assume this condition. We further focus on the reciprocal root which is inside the unit circle and show that it approximates $c_m^{-1}$ for large $m$. Define this single root to be $s_m$ and note, through the reciprocal nature of roots, we have just shown that it satisfies the bound $|s_m| \le |c_m^{-1}|\frac{1}{1-\varepsilon}$, which in turn gives that $|c_m s_m| \le \frac{1}{1-\varepsilon}$. Moreover, $|c_m| > c$ yields the bound $|s_m| \le c^{-1}\frac{1}{1-\varepsilon}$, where $c > 1$ is fixed, and thus choosing $\varepsilon > 0$ small enough we have $|s_m| < r < 1$ for a fixed $r$. This provides the ingredients for the following limit:
	\begin{align*}
		& |s_{m}^{2m+2} - c_{m} s_{m}^{2m+1} + 2(c_{m}-1) s_{m}^{m+1}| \\
		& \le |s_{m}|^{2m+2} + \frac{1}{1-\varepsilon} |s_{m}|^{2m} + 2|s_{m}|^{m+1} + \frac{2}{1-\varepsilon} |s_{m}|^{m} \rightarrow 0
	\end{align*}
	as $m\rightarrow\infty$, since $|s_m| < r < 1$. Now, by definition of $s_m$ as a root of $f_m$, we have that $f_m(s_m) = 0$ and hence we must have that $1 - c_m s_m \rightarrow 0$ and thus $s_m - c_m^{-1} \rightarrow 0$ as $m\rightarrow\infty$, due to $|c_m^{-1}|$ being bounded above by $c^{-1} < 1$. This says that the root which stays inside the unit circle approximates $c_{m}^{-1}$ for large $m$ while the root which stays outside the unit circle must approximate $c_m$ by reciprocal.

	We would now like to interpret what this shows for the potential corresponding root $z_{m}$ in the limit $m\rightarrow\infty$ using the $q$-discriminant condition \eqref{DefiningCurve}. For this we use the definition of the coefficient $c_{m} = {a^{2}}/{z_{m}^{2}}$ from \eqref{q-PolynomialCoeffient} and denote $\delta_{m} = c_{m}s_{m}-1$ where $\delta_{m}\rightarrow0$ as $m\rightarrow\infty$. Then, with $q=s_{m}=c_{m}^{-1}(1+\delta_{m})$, \eqref{DefiningCurve} becomes
	\begin{align}
		\nonumber
		& \frac{B(z_{m})^{2}}{A(z_{m})} = c_{m}^{-1}(1+\delta_{m}) + c_{m}(1+\delta_{m})^{-1} + 2 \\
		\nonumber
		\implies & \frac{(- z_{m}^2 + b^2 - a^2)^{2}}{a^{2}z_{m}^{2}} = \frac{z_{m}^{2}}{a^{2}}(1+\delta_{m}) + \frac{a^{2}}{z_{m}^{2}}(1+\delta_{m})^{-1} + 2 \\
		\label{SingularlyPerturbedEquationForz0}
		\implies & b^{4} - 2 a^2 b^{2} - 2 b^{2} z_{m}^2 = \delta_{m} (z_{m}^{4}-a^4) + \mathcal{O}(\delta_{m}^{2}),
	\end{align}
	where we have used the binomial expansion $(1+\delta_{m})^{-1}=1-\delta_m + \mathcal{O}(\delta_{m}^{2})$, which is valid for large $m$ since $\delta_m \to 0$. Recall that, given we are in case~\ref{Case3}, $|c_m|$ is bounded below away from zero and so $|z_m|$ is bounded above for all $m$. Now note that \eqref{SingularlyPerturbedEquationForz0} is a singular perturbation \cite[Section 7.2]{Bender99} and as $\delta_{m}\rightarrow0$ all possible solutions for $z_m$ go to infinity except for those which satisfy the left-hand side being zero. As such, the only possibility for any $z_m$ being a true root of the characteristic polynomial is that they tend to one of the limiting roots
	\begin{align}
		\label{Limiting-z-different-case}
		z = \pm \sqrt{\tfrac{1}{2}b^{2} - a^{2}}.
	\end{align}
	Note that for such $z_m$ to exist we required the condition $|c_{m}|>1$, and so $|a^{2}|>|z_{m}^{2}|$, to hold for arbitrarily large $m$. For this to hold in the limit we require $|a^{2}|>|\frac{1}{2}b^{2} - a^{2}|$ and so the limiting roots in \eqref{Limiting-z-different-case} may only exist when this condition is met.

	We have now seen that, aside from the special case yielding the potential for limiting roots \eqref{Limiting-z-different-case}, all remaining $z_m$ correspond to $q_m$ values which tend to the unit circle. To complete the proof we now translate this result using the $q$-discriminant condition \eqref{DefiningCurve}. Since $q_m$ tends to the unit circle, the corresponding $z_{m}$ must tend to the limiting curve defined by \eqref{DefiningCurve} where $q = e^{i\phi}$ for some $\phi \in [-\pi,\pi]$. This limiting curve in the complex plane is given parametrically as
	\begin{align*}
		& \frac{B(z)^{2}}{A(z)} = e^{i\phi} + e^{-i\phi} + 2, & & \phi \in [-\pi,\pi] \\
		\iff & \frac{(- z^2 + b^2 - a^2)^{2}}{a^{2}z^{2}} = 4 \cos^{2}\left(\frac{\phi}{2}\right), & & \phi \in [-\pi,\pi] \\
		\iff & z^2 - b^2 + a^2 = \pm 2 az \cos\left(\frac{\phi}{2}\right), & & \phi \in [-\pi,\pi] \\
		\iff & z^2 - 2 a \cos(\theta) z - b^2 + a^2 = 0, & & \theta \in [-\pi,\pi] \\
		\iff & z = a \cos(\theta) \pm \sqrt{b^2 - a^{2} \sin^{2}(\theta)}, & & \theta \in [-\pi,\pi].
	\end{align*}
	Thus, as roots $z_{m}$ of $p_{m}(z)$ are eigenvalues $\lambda$ of $\mathcal{T}\in\mathbb{C}^{2m\times2m}$, we deduce that the limiting spectrum of $\mathcal{T}$ must lie on the curve defined by \eqref{LimitingCurve} as $m\rightarrow\infty$, except perhaps for the eigenvalues in \eqref{LimitingPoints} which can only occur if $|a^{2}|>|\frac{1}{2}b^{2} - a^{2}|$.
\end{proof}

We note that, while the so-called Szeg\H{o} formula does not apply in our non-Hermitian case, we have just proven that the limiting eigenvalues of $\mathcal{T}$, except perhaps two, lie on the equivalent curve defined by eigenvalues of the (block) symbol of $\mathcal{T}$, which is precisely that defined in \eqref{LimitingCurve}.

\section{The one-dimensional problem}
\label{sec:1dcase}

We now turn our attention to analysing the one-level method. In this section we study the parallel Schwarz iterative method for the one-dimensional Maxwell's equations with Robin boundary conditions defined on the domain \mbox{$\Omega = (a_1,b_N)$}:
\begin{align}\label{eq:1d0}
	\left\{\begin{array}{r@{}ll}
		{\cal L}u \vcentcolon= -\partial_{xx}u + (ik\tilde \sigma - k^2)u &{}= 0, & x \in (a_1,b_N),\\
		{\cal B}_lu \vcentcolon= -\partial_x u + \alpha u &{}= g_1, & x = a_1,\\
		{\cal B}_ru \vcentcolon= \partial_x u + \alpha u &{}= g_2, & x = b_N,
	\end{array}\right.
\end{align}
where $u$ represents the complex amplitude of the electric field, $k$ is the wave number, and $\tilde\sigma = \sigma Z$ with $\sigma$ being the conductivity of the medium and $Z$ its impedance. Here $\alpha$ is the impedance parameter which is chosen such that the local problems are well-posed and is classically set to $ik$, in which case the problem corresponds to a {\it ``one-dimensional wave-guide''} and the incoming wave or excitation can be represented by $g_1$, for example, with $g_2$ being set to $0$. Note that, when $\alpha = ik$, the problem is well-posed even if $\tilde\sigma=0$ but in the following we will assume that $\tilde\sigma>0$. In order to simplify notation we will omit the tilde symbol for $\sigma$. We remark that \eqref{eq:1d0} can also be seen as an absorptive Helmholtz equation where the absorption term $ik\sigma$ comes from the physics of the problem.

Let us also consider two sets of points $\{a_j\}_{j=1,..,N+1}$ and $\{b_j\}_{j=0,..,N}$ defining the overlapping decomposition $\Omega = \cup_{j=1}^N \Omega_j$ such that $\Omega_{j} = (a_j,b_j)$, as illustrated in \Cref{fig:decomp1d} (and considered in \cite{Chaouqui:2018:OSC}), where
\begin{align}
	\label{eq:ab}
	b_j-a_j &= L +2\delta, & b_{j-1}-a_j &= 2\delta, & a_{j+1}-a_j &= b_{j+1}-b_j = L, & \delta &> 0.
\end{align}
Note that the length of each subdomain is fixed and equal to $L+2\delta$ while the overlap is always $2\delta$. This means that the family of problems we will consider solving consists of a growing chain of fixed-size subdomains, as in \cite{Chaouqui:2018:OSC}, rather than solving on a fixed problem domain with shrinking subdomain size.

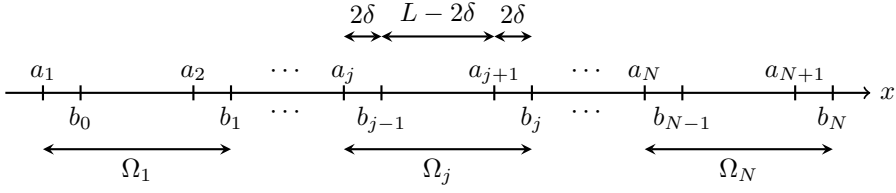
\begin{figure}[t]
	\centering
	\begin{tikzpicture}
		\draw[->,thick] (-0.5,0) -- (11,0) node[right] {$x$};
		\tikzset{>=stealth,thick}
		\foreach \p in {0,0.5,2,2.5,4,4.5,6,6.5,8,8.5,10,10.5}
		{
			\draw[thick] (\p,-0.1) -- (\p,0.1);
		}
		\draw (0,0.5) node[below] {$a_{1}$};
		\draw (2,0.5) node[below] {$a_{2}$};
		\draw (4,0.5) node[below] {$a_{j}$};
		\draw (6,0.5) node[below] {$a_{j+1}$};
		\draw (8,0.5) node[below] {$a_{N}$};
		\draw (10,0.5) node[below] {$a_{N+1}$};
		\draw (0.5,-0.05) node[below] {$b_{0}$};
		\draw (2.5,-0.05) node[below] {$b_{1}$};
		\draw (4.5,-0.05) node[below] {$b_{j-1}$};
		\draw (6.5,-0.05) node[below] {$b_{j}$};
		\draw (8.5,-0.05) node[below] {$b_{N-1}$};
		\draw (10.5,-0.05) node[below] {$b_{N}$};
		\draw (3.25,0.5) node[below] {$\cdots$};
		\draw (3.25,-0.05) node[below] {$\cdots$};
		\draw (7.25,0.5) node[below] {$\cdots$};
		\draw (7.25,-0.05) node[below] {$\cdots$};
		\draw[<->,thick] (0,-0.75) -- (2.5,-0.75);
		\draw[<->,thick] (4,-0.75) -- (6.5,-0.75);
		\draw[<->,thick] (8,-0.75) -- (10.5,-0.75);
		\draw (1.25,-0.75) node[below] {$\Omega_{1}$};
		\draw (5.25,-0.75) node[below] {$\Omega_{j}$};
		\draw (9.25,-0.75) node[below] {$\Omega_{N}$};
		\draw[<->,thick] (4,0.75) -- (4.5,0.75);
		\draw[<->,thick] (4.5,0.75) -- (6,0.75);
		\draw[<->,thick] (6,0.75) -- (6.5,0.75);
		\draw (4.25,0.8) node[above] {$2\delta$};
		\draw (5.25,0.8) node[above] {$L-2\delta$};
		\draw (6.25,0.8) node[above] {$2\delta$};
	\end{tikzpicture}
	\caption{Overlapping decomposition of the one-dimensional domain into $N$ subdomains.}
	\label{fig:decomp1d}
\end{figure}

We consider solving \eqref{eq:1d0} by a Schwarz iterative algorithm with Robin transmission conditions and denote by $u_j^n$ the approximation to the solution in subdomain $j$ at iteration $n$, starting from an initial guess $u_j^0$. We compute $u_j^n$ from the previous values $u_j^{n-1}$ by solving the following local boundary value problem
\begin{subequations}
	\label{eq:sch1dall}
	\begin{align}\label{eq:sch1d}
		\left\{\begin{array}{r@{}ll}
			{\cal L} u_j^n &{}= 0, & x \in \Omega_j,\\
			{\cal B}_l u_j^n &{}= {\cal B}_l u_{j-1}^{n-1}, & x = a_j,\\
			{\cal B}_r u_j^n &{}= {\cal B}_r u_{j+1}^{n-1}, & x = b_j,
		\end{array}\right.
	\end{align}
	in the case $2\le j \le N$ while for the first ($j=1$) and last ($j=N$) subdomain we have
	\begin{align}\label{eq:sch1d1N}
		\begin{array}{cc}
			\left\{\begin{array}{r@{}ll}
				{\cal L}u_1^n &{}= 0, & x \in \Omega_1,\\
				{\cal B}_l u_1^n &{}= g_1, & x = a_1,\\
				{\cal B}_r u_1^n &{}= {\cal B}_r u_{2}^{n-1}, & x = b_1,
			\end{array}\right.
			&
			\left\{\begin{array}{r@{}ll}
				{\cal L}u_N^n &{}= 0, & x \in \Omega_N,\\
				{\cal B}_l u_N^n &{}= {\cal B}_l u_{N-1}^{n-1}, & x = a_N,\\
				{\cal B}_r u_N^n &{}= g_2, & x = b_N.
			\end{array}\right.
		\end{array}
	\end{align}
\end{subequations}
In the following we wish to analyse the convergence of the iterative method that is defined by \eqref{eq:sch1dall}. We observe this iteration to be a parallel Schwarz method with Robin transmission conditions, a label which we shall adopt in this work. In particular, we will be interested in the convergence properties for a growing number of subdomains $N$ and the absorptive problem, i.e., $\sigma>0$.\footnote{When $\sigma=0$, impedance transmission conditions are also transparent conditions, with the resulting iteration matrix being nilpotent. Therefore, the algorithm will converge in a number of iterations equal to the number of subdomains in this case.} This means that we will consider asymptotic bounds for large $N$ and make use of the theory presented in \Cref{sec:toeplitz}.

In order to do this we define the local errors in each subdomain $j$ at iteration $n$ as $e^n_j = u|_{\Omega_j} - u^n_j$. They verify the boundary value problems \eqref{eq:sch1d} for the interior subdomains and the homogeneous analogues of \eqref{eq:sch1d1N} for the first and last subdomains (i.e., \eqref{eq:sch1d1N} but with boundary conditions $g_{1}=0$ and $g_{2}=0$). The convergence study will be done in two steps: first we prove that the Schwarz iteration matrix is a block Toeplitz matrix and then that its spectral radius remains bounded below and away from one in the limit of large $N$. As mentioned before, we build on the formalism of iteration matrices acting on interface data introduced in \cite{Chaouqui:2018:OSC}; here this will be Robin data.

\begin{lemma}[Block Toeplitz iteration matrix]
	\label{lemma:basic1d}
	If $e^n_j = u|_{\Omega_j} - u^n_j$ is the local error in each subdomain $j$ at iteration $n$ and
	\begin{align*}
		{{\cal R}^n} &\vcentcolon= \left[{\cal R}^{n}_+(b_1), \, {\cal R}^{n}_-(a_2), \, {\cal R}^{n}_+(b_2), \ldots, {\cal R}^{n}_-(a_{N-1}), \,{\cal R}^{n}_+(b_{N-1}), \, {\cal R}^{n}_-(a_N)\right]^T,
	\end{align*}
	where
	\begin{align}\label{eq:rpm}
		{\cal R}^{n}_-(a_j) &\vcentcolon= {\cal B}_l e_{j-1}^{n}(a_j), & {\cal R}^{n}_+(b_j) &\vcentcolon={\cal B}_r e_{j+1}^{n}(b_j),
	\end{align}
	is the Robin interface data, then
	\begin{align*}
		{{\cal R}^n} = {\cal T}_{1d} {{\cal R}^{n-1}},
	\end{align*}
	where ${\cal T}_{1d}$ is a block Toeplitz matrix of the form \eqref{NonHermitianBlockToeplitzStructure} with the complex coefficients $a$ and $b$ being given by
	\begin{subequations}
		\label{eq:ab0}
		\begin{align}
			a &= \frac{(\zeta + \alpha )^2 e^{2\zeta \delta} - (\zeta - \alpha)^2 e^{-2\zeta \delta}}{(\zeta + \alpha)^2 e^{\zeta (2\delta+L)} - (\zeta - \alpha)^2 e^{-\zeta (2\delta+L)}}, \\
			b &= -\frac{(\zeta^2-\alpha^2)(e^{\zeta L} - e^{-\zeta L} )}{ (\zeta + \alpha)^2 e^{\zeta (2\delta+L)} - (\zeta - \alpha)^2 e^{-\zeta (2\delta+L)}},
		\end{align}
	\end{subequations}
	where $\zeta = \sqrt{ik\sigma-k^2}$.
\end{lemma}
\begin{proof}
	We first see that the solution to ${\cal L}e_j^n=0$ is given by
	\begin{align}
		\label{eq:ajbj}
		e_j^n(x) &= \alpha_j^n e^{-\zeta x} + \beta_j^n e^{\zeta x}, & \zeta &= \sqrt{ik\sigma-k^2}.
	\end{align}
	Note that we choose the principle branch of the square root here so that $\zeta$ always has positive real and imaginary parts. Now the interface iterations at $x=a_j$ and $x=b_j$ from \eqref{eq:sch1dall} can be written in terms of the error as
	\begin{align}
		\label{eq:blbr}
		\left[\begin{array}{c} {\cal B}_l e_j^n(a_j) \\ {\cal B}_r e_j^n(b_j) \end{array}\right] = \left[\begin{array}{c} {\cal B}_l e_{j-1}^{n-1}(a_j) \\ {\cal B}_r e_{j+1}^{n-1}(b_j)\end{array}\right].
	\end{align}
	By introducing \eqref{eq:ajbj} into the left-hand side of \eqref{eq:blbr} and by using the notation from \eqref{eq:rpm} we obtain
	\begin{align*}
		\left[\begin{array}{cc}
			(\zeta + \alpha) e^{-\zeta a_j} & - (\zeta - \alpha) e^{\zeta a_j} \\
			-(\zeta - \alpha ) e^{-\zeta b_j} & (\zeta + \alpha) e^{\zeta b_j}
		\end{array}
		\right] \left[\begin{array}{c} \alpha_j^n \\ \beta_j^n \end{array}\right] = \left[\begin{array}{c} {\cal R}^{n-1}_-(a_j)\\ {\cal R}^{n-1}_+(b_j)\end{array}\right],
	\end{align*}
	which we can solve for the unknowns $\alpha_j^n$ and $\beta_j^n$ to give
	\begin{align}\label{eq:ajbj0}
		\left[\begin{array}{c} \alpha_j^n \\ \beta_j^n \end{array}\right] & = \frac{1}{D_j} \left[\begin{array}{cc}
			(\zeta+ \alpha) e^{\zeta b_j} & (\zeta - \alpha) e^{\zeta a_j} \\
			(\zeta - \alpha) e^{-\zeta b_j} & (\zeta + \alpha) e^{-\zeta a_j}
		\end{array}
		\right] \left[\begin{array}{c} {\cal R}^{n-1}_-(a_j) \\{\cal R}^{n-1}_+(b_j) \end{array}\right],
	\end{align}
	where $D_j = (\zeta + \alpha)^2 e^{\zeta (b_j-a_j)} - (\zeta - \alpha)^2 e^{\zeta (a_j-b_j)}$. Note that, since $b_j-a_j = L+2\delta$, then $D_{j}$ is actually independent of $j$ and thus we simply denote it by $D$. The algorithm is based on Robin transmission conditions, hence the quantities of interest which are transmitted at the interfaces between subdomains are the Robin data \eqref{eq:rpm}. Therefore, we need to compute the current interface values ${\cal R}^{n}_-(a_j)$ and ${\cal R}^{n}_+(b_j)$ by replacing the coefficients from \eqref{eq:ajbj0} into \eqref{eq:ajbj} and then applying the formulae in \eqref{eq:rpm}, giving
	\begin{subequations}
		\label{eq:rjajbj}
		\begin{align}
			\begin{split}
				{\cal R}^{n}_-(a_j) & = {\cal B}_l e_{j-1}^{n}(a_j) = (\zeta +\alpha) \alpha^n_{j-1}e^{-\zeta a_j} - (\zeta - \alpha) \beta^n_{j-1}e^{\zeta a_j} \\
				& = \frac{1}{D} \Bigl[ ((\zeta + \alpha)^2 e^{\zeta (b_{j-1}-a_j)} - (\zeta - \alpha)^2 e^{\zeta (a_j-b_{j-1})}){\cal R}^{n-1}_-(a_{j-1}) \Bigr.\\
				& \mathrel{\phantom{=}} \mathrel+ \Big. (\zeta^2-\alpha^2)( e^{\zeta (a_{j-1}-a_{j})}-e^{\zeta (a_{j}-a_{j-1})} ) {\cal R}^{n-1}_+(b_{j-1}) \Bigr],
			\end{split} \\
			\begin{split}
				{\cal R}^{n}_+(b_j) & = {\cal B}_r e_{j+1}^{n}(b_j) = -(\zeta -\alpha) \alpha^n_{j+1}e^{-\zeta b_j} + (\zeta +\alpha) \beta^n_{j+1}e^{\zeta b_j} \\
				& = \frac{1}{D} \Bigl[ (\zeta^2-\alpha^2)(e^{\zeta (b_{j}-b_{j+1})}-e^{\zeta (b_{j+1}-b_{j})} )
				{\cal R}^{n-1}_-(a_{j+1}) \Bigr. \\
				& \mathrel{\phantom{=}} \mathrel+ \Bigl. ((\zeta + \alpha)^2 e^{\zeta (b_j-a_{j+1})} - (\zeta - \alpha)^2 e^{\zeta (a_{j+1}-b_j)}) {\cal R}^{n-1}_+(b_{j+1}) \Bigr].
			\end{split}
		\end{align}
	\end{subequations}
	The iteration of interface values \eqref{eq:rjajbj} can be summarised as follows:
	\begin{subequations}
		\label{eq:rniterall}
		\begin{align}
			\label{eq:rniter}
			\begin{split}
				\left[\begin{array}{c} {\cal R}^{n}_-(a_j) \\ {\cal R}^{n}_+(b_j) \end{array}\right] & = T_1 \left[\begin{array}{c} {\cal R}^{n-1}_-(a_{j-1}) \\ {\cal R}^{n-1}_+(b_{j-1}) \end{array}\right] + T_2 \left[\begin{array}{c} {\cal R}^{n-1}_-(a_{j+1}) \\ {\cal R}^{n-1}_+(b_{j+1}) \end{array}\right], \\ T_1 & = \left[\begin{array}{cc} a & b \\ 0 & 0 \end{array} \right], \ T_2 = \left[\begin{array}{cc} 0 & 0 \\ b & a \end{array} \right],
			\end{split}
		\end{align}
		where $a$ and $b$ are given by \eqref{eq:ab0}. Note that since the homogeneous counterparts of the boundary conditions from \eqref{eq:sch1d1N} translate into ${\cal R}^{n}_-(a_1)=0$ and ${\cal R}^{n}_+(b_N)=0$ for all $n$, we can remove these terms. As such, the iterates for $j\in \{1,2,N-1,N\}$ are prescribed slightly differently as
		\begin{align}
			\label{eq:rniterends}
			\begin{split}
				\left[\begin{array}{c} 0 \\ {\cal R}^{n}_+(b_1) \end{array}\right] &= T_2 \left[\begin{array}{c} {\cal R}^{n-1}_-(a_{2}) \\ {\cal R}^{n-1}_+(b_{2}) \end{array}\right], \\
				\left[\begin{array}{c} {\cal R}^{n}_-(a_2) \\ {\cal R}^{n}_+(b_2) \end{array}\right] &= T_1 \left[\begin{array}{c} 0 \\{\cal R}^{n-1}_+(b_1) \end{array}\right] + T_2 \left[\begin{array}{c} {\cal R}^{n-1}_-(a_3) \\ {\cal R}^{n-1}_+(b_3) \end{array}\right], \\
				\left[\begin{array}{c} {\cal R}^{n}_-(a_{N-1}) \\ {\cal R}^{n}_+(b_{N-1}) \end{array}\right] &= T_1 \left[\begin{array}{c} {\cal R}^{n-1}_-(a_{N-2}) \\ {\cal R}^{n-1}_+(b_{N-2}) \end{array}\right] + T_2 \left[\begin{array}{c} {\cal R}^{n-1}_-(a_N) \\ 0\end{array}\right], \\
				\left[\begin{array}{c} {\cal R}^{n}_-(a_N) \\ 0 \end{array}\right] &= T_1 \left[\begin{array}{c} {\cal R}^{n-1}_-(a_{N-1}) \\ {\cal R}^{n-1}_+(b_{N-1}) \end{array}\right].
			\end{split}
		\end{align}
	\end{subequations}
	With the notation from \eqref{eq:rpm}, global iteration over interface data belonging to all subdomains becomes ${{\cal R}^n} = {\cal T}_{1d} {{\cal R}^{n-1}}$ where
	\begin{align}
		\label{eq:itert1}
		{{\cal T}_{1d}} = \left[\begin{array}{ccccccc}
			0 & \widehat T_2 & & & & & \\
			\widetilde T_1 & 0_{2\times 2} & T_2 & & & & \\
			& \ddots & \ddots & \ddots & & & \\
			& & T_1 &0_{2\times 2}& T_2 & &\\
			& & & \ddots & \ddots & \ddots & \\
			& & & & T_1 & 0_{2\times 2}& \widetilde T_2 \\
			& & & & & \widehat T_1 & 0
		\end{array}\right]
	\end{align}
	with $\widetilde T_1 = \left[\begin{array}{cc} b & 0 \end{array} \right]^T$, $\widetilde T_2 = \left[\begin{array}{cc} 0 & b \end{array} \right]^T$, $\widehat T_1 = \left[\begin{array}{cc} a & b \end{array} \right]$, $\widehat T_2 = \left[\begin{array}{cc} b & a \end{array} \right]$. We conclude from this that the parallel Schwarz algorithm is given by a stationary iteration with iteration matrix ${\cal T}_{1d}$ defined by \eqref{eq:itert1} and, therefore, convergence is determined by the spectral radius $\rho({\cal T}_{1d})$. We also notice that ${\cal T}_{1d}$ is a block Toeplitz matrix precisely of the form in \eqref{NonHermitianBlockToeplitzStructure} where the complex coefficients $a$ and $b$ are given by \eqref{eq:ab0} and, as such, the limiting spectral analysis in \Cref{sec:toeplitz} will apply.
\end{proof}

Before proving convergence of the parallel Schwarz algorithm, we first utilise the key result of \Cref{theorem:LimitingSpectrum}, on the limiting spectrum of $\mathcal{T}_{1d}$, to provide a useful intermediary lemma. This intermediary result will also aid our analysis in the two-dimensional case to follow in \Cref{sec:2dcase}.
\begin{lemma}[Limiting spectral radius and sufficient conditions for convergence]
	\label{lemma:LimitingSpectralRadius1D}
	The following relation holds:
	\begin{align*}
		\max_{\theta \in [-\pi,\pi]}\left|a \cos(\theta) \pm \sqrt{b^2 - a^2 \sin^2(\theta)}\right| = \max\{ |a + b|, | a -b|\},
	\end{align*}
	and thus the convergence factor $R_{1d} \vcentcolon= \lim_{N\rightarrow\infty} \rho({\cal T}_{1d})$ of the Schwarz algorithm as the number of subdomains tends to infinity verifies
	\begin{align}
		\label{eq:r1d}
		R_{1d} \le
		\left\{\begin{array}{ll}
			\max \left\lbrace |a + b|, | a - b| \right\rbrace & \text{if } \left|a^2-\frac{1}{2}b^2\right|^{1/2} \ge |a|, \\
			\max \left\lbrace |a + b|, | a - b|, |a| \right\rbrace & \text{if } \left|a^2-\frac{1}{2}b^2\right|^{1/2} < |a|.
		\end{array}\right.
	\end{align}
	Further, consider the change of variables
	\begin{align}
		\label{eq:params}
		z &= 2\delta\zeta, & l &= \frac{L}{2\delta}, & \gamma &= 2\delta \alpha, & v &= \frac{z-\gamma}{z+\gamma},
	\end{align}
	and let $z \vcentcolon= x+iy$ for $x,y \in \mathbb{R}^{+}$. Then the condition $g_{\pm}(z;\delta,l) > 0$, where
	\begin{align}
		\label{eq:g}
		g_{\pm}(z;\delta,l) &= (e^{2lx} - 1)(e^{2x} - |v|^2) \pm 4 \sin(ly) (\Im v\cos{y} - \Re v \sin{y}) e^{x(l+1)},
	\end{align}
	will ensure the desired convergence bound $\max\{ |a + b|, | a - b| \} <1$. Similarly, the condition $g(z;\delta,l) > 0$, where
	\begin{align}
		\label{eq:ga}
		\begin{split}
			g(z;\delta,l) &= (e^{2lx} - 1)(e^{2x(l+2)} - |v|^4) + 4 \sin(ly) \\
			& \mathrel{\phantom{=}} \mathrel\cdot \left[((\Re v)^2-(\Im v)^2) \sin(y(l+2)) - 2 \Re v \Im v \cos(y(l+2))\right]e^{2x(l+1)},
		\end{split}
	\end{align}
	will ensure that $|a|<1$.
\end{lemma}
\begin{proof}
	Since $\mathcal{T}_{1d}$ is of the form $\mathcal{T}$ in \eqref{NonHermitianBlockToeplitzStructure}, \Cref{theorem:LimitingSpectrum} provides its limiting spectrum and thus allows us to bound $R_{1d}$ by the largest eigenvalue in magnitude. We first bound $\lambda_{\pm}(\theta) = a \cos(\theta) \pm \sqrt{b^2 - a^2 \sin^2(\theta)}$. It is straightforward to see that these values
	are the eigenvalues of the matrix
	\begin{align*}
		T = \left(\begin{array}{cc}
			a \cos(\theta) & b - a\sin(\theta) \\
			b + a\sin(\theta) & a \cos(\theta)
		\end{array}
		\right).
	\end{align*}
	A simple computation shows that the matrix
	\begin{align*}
		T^*T = \left(\begin{array}{cc}
			|a|^2 +|b|^2 +(a\bar b+\bar ab)\sin(\theta) & (a\bar b+\bar ab)\cos(\theta)\\
			(a\bar b+\bar ab)\cos(\theta) & |a|^2 +|b|^2 - (a\bar b+\bar ab)\sin(\theta)
		\end{array}
		\right)
	\end{align*}
	has the eigenvalues $\mu_{\pm} = |a\pm b|^2$. We can now conclude that
	\begin{align*}
		|\lambda_{\pm}(\theta)| \le \|T\|_2 = \sqrt{\|T^*T\|_2} = \sqrt{\max\{\mu_+,\mu_-\} } = \max\{|a+b|,|a-b|\},
	\end{align*}
	and furthermore note that this bound is attained when $\theta = 0$. Additionally, \Cref{theorem:LimitingSpectrum} states that eigenvalues $\lambda = \pm(\tfrac{1}{2}b^{2} - a^{2})^{1/2}$ may belong to the limiting spectrum but only if they have magnitude strictly less than $|a|$. Together, these two cases yield \eqref{eq:r1d}.

	Let us consider now the complex-valued functions $F_{\pm} \colon \mathbb{C} \rightarrow \mathbb{C}$
	\begin{align*}
		F_{\pm}(z) = \frac{(z + \gamma)^2 e^{z} - (z- \gamma )^2 e^{-z}}{(z + \gamma)^2 e^{(l+1)z} - (z - \gamma)^2 e^{-(l+1)z}} \pm \frac{(z^2-\gamma^2)(e^{lz} -e^{-lz} ) }{(z + \gamma)^2 e^{(l+1)z} - (z-\gamma)^2 e^{-(l+1)z}}.
	\end{align*}
	It is easy to see that ${a \mp b} = F_{\pm} (z)$ when $z$, $l$ and $\gamma$ are as defined in \eqref{eq:params}. Similarly, we define the function $G \colon \mathbb{C} \to \mathbb{C}$ to be the first term in $F_{\pm}(z)$ so that $a = G(z)$. Let us simplify in the first instance the expression of $|F_{\pm}(z)|$ without using any assumption on $z \vcentcolon= x+iy$. For this we consider the transformation $v$ along with its polar form
	\begin{align}\label{eq:v}
		v &\vcentcolon= \frac{z-\gamma}{z+\gamma}, & v &= w(\cos(\varphi)+i\sin (\varphi)), & w=|v|.
	\end{align}
	After some lengthy but elementary calculations we find that
	\begin{subequations}
		\label{eq:fpmandgpm}
		\begin{align}
			\label{eq:fpm}
			|F_{\pm}(z)|^2 &= 1 - \frac{(e^{x(l + 1)}-w)^2+2w(1\mp\cos\left({(l + 1)y} - \varphi\right))e^{x(l + 1)} }{(e^{2x(l + 1)}-w^2)^2+4w^2\sin^2((l + 1)y - \varphi)e^{2x(l + 1)} } g_{\pm}(z;\delta,l)\\
			\begin{split}
				\label{eq:fpmgpm}
				g_{\pm}(z;\delta,l) &= (e^{2lx} - 1)(e^{2x} - w^2) \pm 4 w\sin(ly)\sin(\varphi - y)e^{x(l + 1)}.
			\end{split}
		\end{align}
	\end{subequations}
	We observe that the fraction in \eqref{eq:fpm} is positive, since the individual terms involved are, and thus $\max\{|a-b|,|a+b|\} < 1 \Leftrightarrow |F_{\pm}(z)|^2 < 1 \Leftrightarrow g_{\pm}(z;\delta,l) > 0$. We can now rewrite $g_{\pm}(z;\delta,l)$ in \eqref{eq:fpmgpm} using \eqref{eq:v} and convert $v$ to Cartesian form to obtain the required expression in \eqref{eq:g}. A near identical argument can be used to derive conditions for $|G(z)|^2 < 1$ and results in the criterion that $g(z;\delta,l) > 0$, where $g(z;\delta,l)$ is defined by \eqref{eq:ga}. Thus the required conclusions follow.
\end{proof}

We are now ready to state our main convergence result for the one-dimensional problem in the case when $\alpha = ik$, namely that of classical impedance conditions.
\begin{theorem}[Convergence of the Schwarz algorithm in 1D]
	\label{theorem:1d}
	If $\alpha = ik$ (the case of classical impedance conditions), then for all $k>0$, $\sigma>0$, $\delta>0$ and $L>0$ we have that $R_{1d} < 1$. Therefore the convergence will ultimately be independent of the number of subdomains (we say that the Schwarz method will scale).
\end{theorem}

\begin{proof}
	By \Cref{lemma:LimitingSpectralRadius1D} we see that it is enough to study the sign of $g_{\pm}(z;\delta,l)$ and of $g(z;\delta,l)$. We can see that if $\alpha = ik$ and $\kappa = 2\delta k$ then for $z\vcentcolon= x+iy$ \eqref{eq:v} becomes
	\begin{align*}
		\Re v &= \frac{-\kappa^2 + x^2 + y^2}{(\kappa+y)^2 + x^2}, & \Im v &= \frac{-2\kappa x}{(\kappa+y)^2 + x^2}, & |v|^2 &= \frac{(\kappa-y)^2 + x^2}{(\kappa+y)^2 + x^2} < 1,
	\end{align*}
	the final inequality holding since $\kappa > 0$ and $y > 0$. We emphasise that $x$ and $y$ are the real and imaginary parts of $z=2\delta\zeta$ and so are positive by the nature of $\zeta$ in \eqref{eq:ajbj}. Now we can further simplify \eqref{eq:g} using these expressions for $v$ to obtain
	\begin{subequations}
		\label{eq:gpm}
		\begin{align}
			\label{eq:gpmg}
			g_{\pm}(z;\delta,l) &= \frac{4e^{x(l+1)}}{(\kappa+y)^2 + x^2} \tilde{g}_{\pm}(z;\delta,l) \\
			\begin{split}
				\label{eq:gpmgt}
				\tilde{g}_{\pm}(z;\delta,l) &= [(\kappa^2+x^2+y^2) \sinh(x) + 2\kappa y \cosh(x)]\sinh(lx) \\
				& \mathrel{\phantom{=}} \mathrel\pm [(\kappa^2-x^2-y^2) \sin(y) - 2\kappa x \cos(y)]\sin(ly).
			\end{split}
		\end{align}
	\end{subequations}
	Proving positivity of $g_{\pm}(z;\delta,l)$ is then equivalent to positivity of $\tilde g_{\pm}(z;\delta,l)$. To proceed we relate $x$ and $y$ by considering the real part of $z^2 = (x+iy)^2 = 2i\kappa\delta\sigma - \kappa^2$ which yields $y^2 = \kappa^2 + x^2$. Let us now eliminate $y$ using this identity to obtain
	\begin{align*}
		\begin{split}
			\tilde{g}_{\pm}(z;\delta,l) &= 2 \left[(\kappa^2+x^2) \sinh(x) + \kappa \sqrt{\kappa^2+x^2} \cosh(x)\right]\sinh(lx) \\
			& \mathrel{\phantom{=}} \mathrel\mp 2 \left[x^2 \sin(\sqrt{\kappa^2+x^2}) + \kappa x \cos(\sqrt{\kappa^2+x^2})\right]\sin(l\sqrt{\kappa^2+x^2}).
		\end{split}
	\end{align*}
	To show that this is positive we want to lower bound the hyperbolic term in the first line (which is positive) while making the trigonometric term in the second line as large as possible in magnitude and negative. To do this we make use of some elementary bounds which hold for $t > 0$:
	\begin{align}
		\label{eq:HypTrigBounds}
		|\sin(t)| &< t < \sinh(t), & |\cos(t)| &\le 1 < \cosh(t).
	\end{align}
	We can now derive the positivity bound on $\tilde{g}_{\pm}(z;\delta,l)$, noting that $x>0$, as follows
	\begin{align*}
		\tilde{g}_{\pm}(z;\delta,l) &> 2 \left[(\kappa^2+x^2)x + \kappa \sqrt{\kappa^2+x^2}\right] lx - 2 \left[x^2 \sqrt{\kappa^2+x^2} + \kappa x \right] l\sqrt{\kappa^2+x^2} = 0.
	\end{align*}
	Turning to $g(z;\delta,l)$, we can follow a similar process, simplifying \eqref{eq:ga} to find that
	\begin{subequations}
		\label{eq:g1d}
		\begin{align}
			\label{eq:g1dg}
			g(z;\delta,l) &= \frac{4e^{2x(l+1)}}{((\kappa+y)^2 + x^2)^2} \tilde{g}(z;\delta,l) \\
			\begin{split}
				\label{eq:g1dgt}
				\tilde{g}(z;\delta,l) &= \bigl[ ((\kappa^2+x^2+y^2)^2+4\kappa^2y^2) \sinh(x(l+2)) \\
				& \qquad\qquad\;\ \mathrel+ 4\kappa y(\kappa^2+x^2+y^2) \cosh(x(l+2)) \bigr] \sinh(lx) \\
				& \mathrel{\phantom{=}} \mathrel+ \bigl[ ((-\kappa^2+x^2+y^2)^2-4\kappa^2x^2) \sin(y(l+2)) \\
				& \qquad\qquad\;\ \mathrel+ 4\kappa x(-\kappa^2+x^2+y^2) \cos(y(l+2)) \bigr] \sin(ly).
			\end{split}
		\end{align}
	\end{subequations}
	Using the identity $y^2 = \kappa^2 + x^2$ along with the elementary bounds \eqref{eq:HypTrigBounds} we obtain
	\begin{align*}
		\begin{split}
			\tilde{g}(z;\delta,l) &= 4 \left[ y^2(y^2+\kappa^2) \sinh(x(l+2)) + 2\kappa y^3 \cosh(x(l+2)) \right] \sinh(lx) \\
			& \mathrel{\phantom{=}} \mathrel+ 4 \left[ x^2(x^2-\kappa^2) \sin(y(l+2)) + 2 \kappa x^3 \cos(y(l+2)) \right] \sin(ly)
		\end{split} \\
		&> 4 \left[ y^2(y^2+\kappa^2) x(l+2) + 2\kappa y^3 \right] lx - 4 \left[ x^2(x^2+\kappa^2) y(l+2) + 2 \kappa x^3 \right] ly \\
		&= 4l(l+2)x^2y^2\kappa^2 + 8lxy\kappa^3 \\
		&> 0.
	\end{align*}
	Thus, we conclude that for any choice of parameters the required sufficient criteria from \Cref{lemma:LimitingSpectralRadius1D} on $g_{\pm}(z;\delta,l)$ and $g(z;\delta,l)$ hold and hence $R_{1d} < 1$. Therefore the algorithm will always converge in a number of iterations ultimately independent of the number of subdomains. Nonetheless, note that as any problem parameter shrinks to zero the bounds become tight and so $R_{1d}$ can be made arbitrarily close to one.
\end{proof}

\begin{figure}
	\centering
	\hfill
	\includegraphics[width=0.4\textwidth,trim=10cm 0.8cm 10cm 1.6cm ,clip]{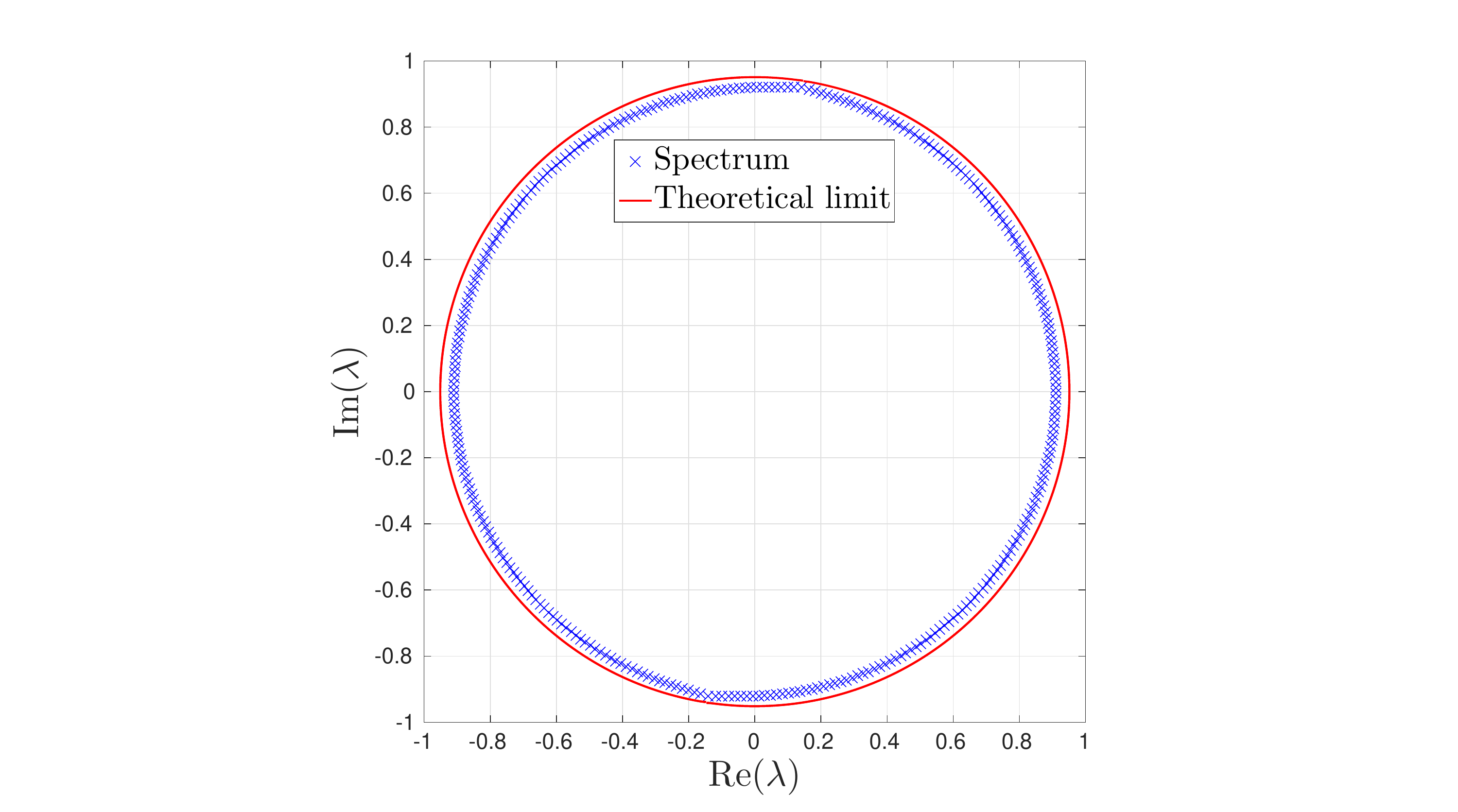}
	\hfill
	\includegraphics[width=0.4\textwidth,trim=10cm 0.8cm 10cm 1.6cm ,clip]{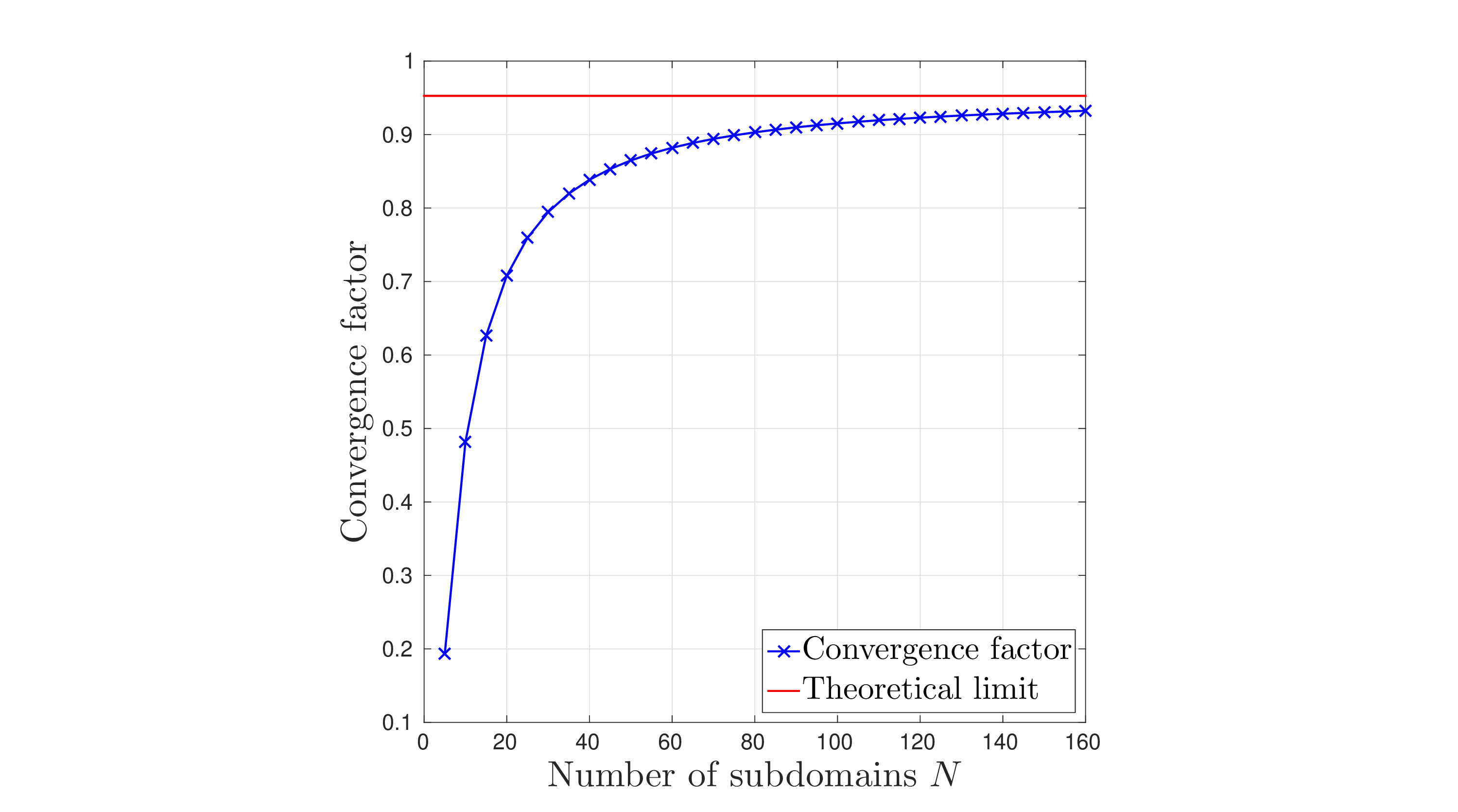}
	\hfill
	\caption{The spectrum of the iteration matrix ${\cal T}_{1d}$ for $N=160$ (left) and the convergence factor of the Schwarz algorithm for varying number of subdomains $N$ (right) when $\sigma = 0.1$.}
	\label{fig:s01}
\end{figure}

\begin{figure}
	\centering
	\hfill
	\includegraphics[width=0.4\textwidth,trim=10cm 0.8cm 10cm 1.6cm ,clip]{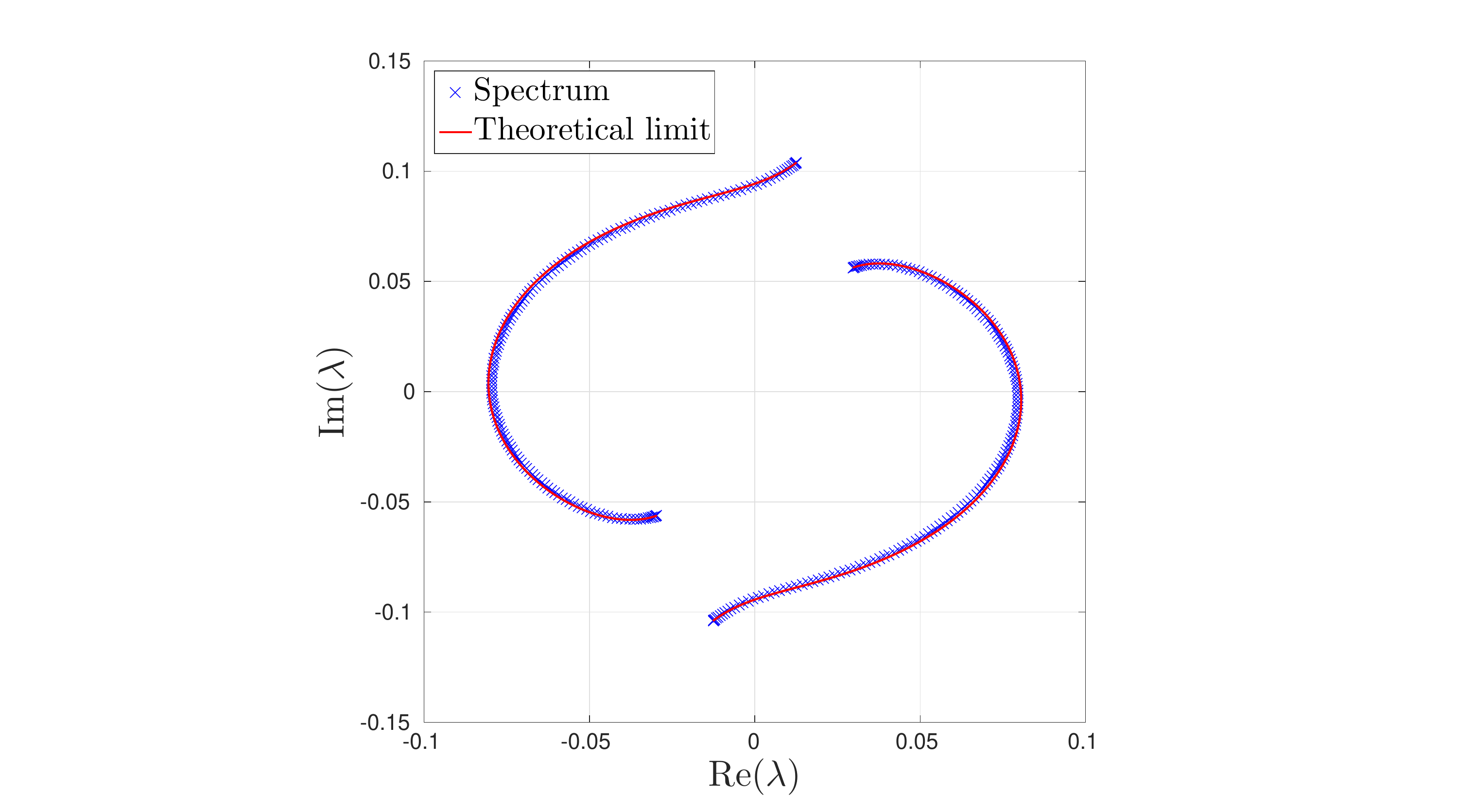}
	\hfill
	\includegraphics[width=0.4\textwidth,trim=10cm 0.8cm 10cm 1.6cm ,clip]{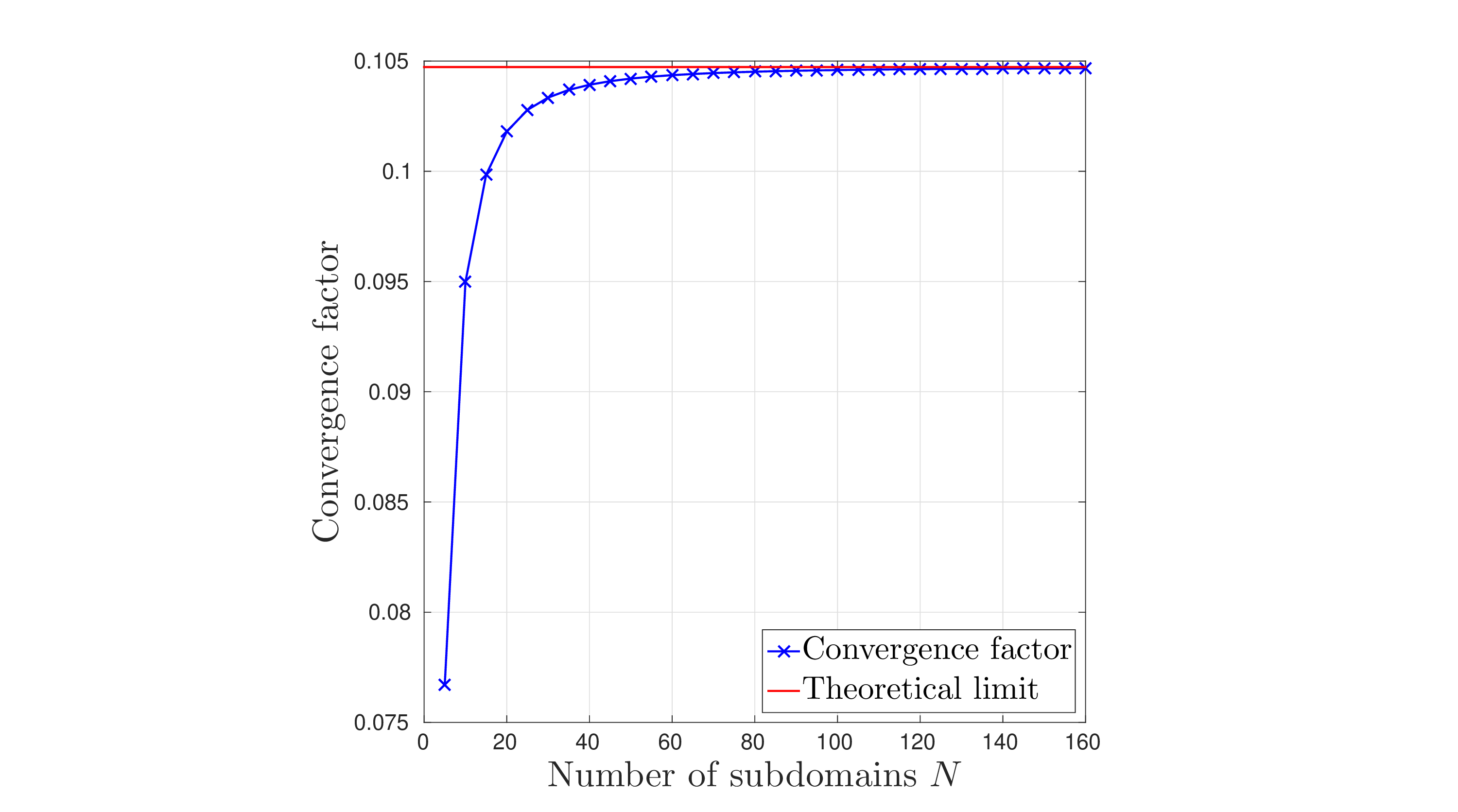}
	\hfill
	\caption{The spectrum of the iteration matrix ${\cal T}_{1d}$ for $N=160$ (left) and the convergence factor of the Schwarz algorithm for varying number of subdomains $N$ (right) when $\sigma = 5$.}
	\label{fig:s5}
\end{figure}

In order to verify this result, we compute numerically (using \texttt{MATLAB}) the spectrum of the iteration matrix and compare it with the theoretical limit for different values of $\sigma$. We choose here $k = 30$, $L=1$ and $\delta = L/10$. From \Cref{fig:s01,fig:s5} we notice that the spectrum of the iteration matrix tends to the theoretical limit when the number of subdomains becomes large and the algorithm remains convergent. Additionally, when $\sigma$ grows the behaviour of the algorithm improves, which is consistent with the fact that when the absorption in the equations is important (solutions are less oscillatory) or the overlap is large (more information is exchanged) the systems are easier to solve. We also remark an empirical observation that the convergence factor monotonically increases towards the limit given in \Cref{lemma:LimitingSpectralRadius1D}, thus indicating that the algorithm will always converge for any $N$.

Before moving onto the two-dimensional case, we first derive a simple corollary showing how our results can be extended in the direction of $k$-independence of the one-level method within certain scenarios. In this case we consider the parameters $L$ and $\delta$ being dependent upon the wave number $k$.
\begin{corollary}[A case of $k$-independent convergence]
	\label{corollary:k-indpendence1D}
	Suppose $\alpha = ik$ (the case of classical impedance conditions) and that $\sigma = \sigma_{0} k$ for some constant $\sigma_{0}$. Consider a $k$-dependent domain decomposition given by $L = L_{0} k^{-1}$ and $\delta = \delta_{0} k^{-1}$, that is the subdomain size and overlap shrink inversely proportional to the wave number. Then the convergence of the corresponding Schwarz method is independent of the wave number $k$. Thus the approach is $k$-robust and convergence will ultimately be independent of the number of subdomains.
\end{corollary}
\begin{proof}
	Inserting the relevant $k$-dependent parameters $\alpha$, $\sigma$, $\delta$ and $L$ into \eqref{eq:ab0} we find that both coefficients $a$ and $b$, and thus the iteration matrix ${\cal T}_{1d}$, are independent of $k$. Combining this result with \Cref{theorem:1d} shows that the convergence of the corresponding Schwarz method is both $k$-independent and, ultimately, independent of the number of subdomains.
\end{proof}

We note that $k$-robustness of the one-level method was proved, under certain conditions, in \cite{Graham:2020:DDI} using rigorous GMRES bounds. Here, our theory is able to directly evidence $k$-robustness of the algorithm at the continuous level, independent of the discretisation, in a simple one-dimensional scenario. We can also consider the case where $k$ is linked to $N$ such that we now solve on a fixed domain a family of problems with increasing wave number using an increasing number of subdomains, here our theory shows the method to be $k$-robust and weakly scalable.

\Cref{theorem:1d} shows that weak scalability is achieved in the one-dimensional case as soon as the parameter $\sigma$ is strictly positive. Intuitively this makes sense since, in the one-dimensional case for $\sigma=0$, impedance conditions are transparent and therefore a classical iterative method will need a number of iterations equal to the number of subdomains to converge (hence no scalability). The complex shift brought about by $\sigma$ will aid convergence by damping the waves and, when this damping parameter is large enough, robustness with respect to the wave number can also be achieved as seen in \Cref{corollary:k-indpendence1D}.

\section{The two-dimensional problem}
\label{sec:2dcase}

Consider the domain $\Omega = (a_1,b_N)\times (0,\hat L)$ on which we wish to solve the two-dimensional problem and a decomposition into $N$ overlapping subdomains defined by $\Omega_j = (a_j,b_j)\times (0,\hat L)$, where $a_j$ and $b_j$ are as given in \eqref{eq:ab}. We will analyse the case of the Helmholtz equation and then Maxwell's equations.

\subsection{The Helmholtz equation}
\label{sec:2dhelmholtz}
The definition of the parallel Schwarz method with Robin transmission conditions for the iterates $u_j^n$ in the case of the two-dimensional Helmholtz problem is
\begin{align}
	\label{eq:2dh}
	\left\{ \begin{array}{r@{}ll}
		(ik\sigma-k^2) u_j^{n} - (\partial_{xx} + \partial_{yy}) u_j^{n} &{}= f, & (x,y) \in (a_j,b_j) \times (0,\hat{L}), \\
		{\cal B}_l u_j^{n}(a_j,y) &{}= {\cal B}_l u_{j-1}^{n-1}(a_j,y), & y \in (0,\hat L), \\
		{\cal B}_r u_j^{n}(b_j,y) &{}= {\cal B}_r u_{j+1}^{n-1}(b_j,y), & y \in (0,\hat L), \\
		u_j^n(x,y) &{}= 0, & x \in (a_j,b_j), \ y \in \lbrace0,\hat L\rbrace,
	\end{array} \right.
\end{align}
where the boundary operators ${\cal B}_l$ and ${\cal B}_r$ are as defined in \eqref{eq:1d0}. We consider here the case of impedance conditions, i.e., $\alpha = ik$. Note that this configuration corresponds to a {\it ``two-dimensional wave-guide''} problem. By linearity, it follows that the local errors $e_j^{n} = u|_{\Omega_j} - u_j^n$ satisfy the homogeneous analogue of \eqref{eq:2dh}. To proceed, we make use of the Fourier sine expansion of $e_j^{n}$, as the solution verifies Dirichlet boundary conditions on the top and bottom of each rectangular subdomain:
\begin{align}
	\label{eq:Fourier}
	e_j^{n}(x,y) &= \sum_{m=1}^{\infty} v_j^{n}(x,\tilde{k}) \sin(\tilde{k}y), & \tilde{k} &= \frac{m\pi}{\hat{L}}, \ m \in \mathbb{N}.
\end{align}
Inserting this expression into the homogeneous counterpart of \eqref{eq:2dh} we find that, for each Fourier number $\tilde k$, $v_j^{n}(x, \tilde k)$ verifies the one-dimensional problem
\begin{align}
	\label{eq:2dv}
	\left\{ \begin{array}{r@{}ll}
		(ik\sigma+\tilde k^2-k^2)v_j^{n} - \partial_{xx}v_j^{n} &{}= 0, & x \in (a_j,b_j), \\
		{\cal B}_l v_j^{n}(x,\tilde k) &{}= {\cal B}_l v_{j-1}^{n-1}(x,\tilde k), & x = a_j,\\
		{\cal B}_r v_j^{n}(x,\tilde k) &{}= {\cal B}_r v_{j+1}^{n-1}(x,\tilde k), & x = b_j,
	\end{array} \right.
\end{align}
which is of exactly the same type as \eqref{eq:sch1dall} where $ik\sigma-k^2$ is replaced by $ik\sigma+\tilde k^2-k^2$. Therefore, the result from \Cref{lemma:basic1d} applies here if we replace $\alpha$ with $ik$ and $\zeta$ with
\begin{align}
	\label{eq:lamk}
	\zeta(\tilde k) = \sqrt{ik\sigma+\tilde k^2-k^2}.
\end{align}
Let us denote the resulting iteration matrix, which propagates information for each Fourier number $\tilde k$ independently, by ${\cal T}_{1d}^{\mathrm{H}}(\tilde k)$ and further let $R_{1d}^{\mathrm{H}}(\tilde k) \vcentcolon= \lim_{N\rightarrow\infty} \rho({\cal T}_{1d}^{\mathrm{H}}(\tilde k))$ with $R_{2d}^{\mathrm{H}} = \sup_{\tilde k} R_{1d}^{\mathrm{H}}(\tilde k)$. We can now state our main convergence result for the two-dimensional Helmholtz problem.

\begin{theorem}[Convergence of the Schwarz algorithm for Helmholtz in 2D]
	\label{theorem:2dhelmholtz}
	If $\alpha = ik$ (the case of classical impedance conditions), then for all $k>0$, $\sigma>0$, $\delta>0$ and $L>0$ we have that $R_{1d}^{\mathrm{H}}(\tilde k) < 1$ for all evanescent modes $\tilde k > k$. Furthermore, under the assumption that between them $\sigma$, $\delta$ and $L$ are sufficiently large we have that $R_{2d}^{\mathrm{H}} < 1$. In particular, this is true when $\sigma \ge k$ for all $\delta>0$ and $L>0$. Therefore the convergence will ultimately be independent of the number of subdomains (we say that the Schwarz method will scale).
\end{theorem}

\begin{proof}
	By \Cref{lemma:LimitingSpectralRadius1D} we see that it is enough to study the sign of $g_{\pm}(z;\delta,l)$ and $g(z;\delta,l)$. To assist, we use the scaled notation $\kappa = 2\delta k$, $\tilde \kappa = 2\delta \tilde k$ and $s = 2\delta\sigma$ akin to \eqref{eq:params}. Now $g_{\pm}(z;\delta,l)$ can be formally simplified identically to \eqref{eq:gpm}, however, in this case with $\zeta$ as in \eqref{eq:lamk} the real part of $z^2$ gives the identity $\tilde \kappa^2 - \kappa^2 = x^2 - y^2$. Utilising this identity along with the bounds \eqref{eq:HypTrigBounds} yields
	\begin{align*}
		\tilde{g}_{\pm}(z;\delta,l) &> \left[(\kappa^2+x^2+y^2)x + 2\kappa y\right] lx - \left|(\kappa^2-x^2-y^2)y -2\kappa x\right| ly \\
		&\ge l (\kappa^2+x^2+y^2) (\tilde \kappa^2-\kappa^2).
	\end{align*}
	Hence we always have $\tilde{g}_{\pm}(z;\delta,l) > 0$ for the evanescent modes $\tilde k > k$ (equivalent to $\tilde \kappa > \kappa$). Similarly, $g(z;\delta,l)$ can be simplified identically to \eqref{eq:g1d} and we find that
	\begin{align*}
		\begin{split}
			\tilde{g}(z;\delta,l) &> l(l+2) \left( x^2((\kappa^2+x^2+y^2)^2+4\kappa^2y^2) - y^2|(-\kappa^2+x^2+y^2)^2-4\kappa^2x^2| \right) \\
			& \mathrel{\phantom{=}} \mathrel+ 4l\kappa xy \left( \kappa^2+x^2+y^2 - |-\kappa^2+x^2+y^2| \right)
		\end{split} \\
		&\ge l(l+2) (\kappa^2+x^2+y^2)^2 (\tilde \kappa^2-\kappa^2),
	\end{align*}
	and so we always have $\tilde{g}(z;\delta,l) > 0$ for the evanescent modes $\tilde k > k$ too. Together this shows that $R_{1d}^{\mathrm{H}}(\tilde k) < 1$ for all evanescent modes. Note that, for the remaining modes $\tilde k \le k$, it is possible that $R_{1d}^{\mathrm{H}}(\tilde k) \ge 1$ for some choices of problem parameters.

	We now refine the above bounds. In order to do so we make use of the identities $4x^2y^2 = \kappa^2s^2$ and $x^2+y^2 = \sqrt{(\tilde \kappa^2-\kappa^2)^2 + \kappa^2 s^2}$ which arise since (by considering both real and imaginary parts of $z^2 = (x+iy)^2 = i\kappa s+\tilde \kappa^2-\kappa^2$) we have that
	\begin{align}
		\label{eq:2x^2and2y^2}
		2x^2 & = \sqrt{(\tilde \kappa^2-\kappa^2)^2 + \kappa^2 s^2} + \tilde \kappa^2 - \kappa^2, & 2y^2 & = \sqrt{(\tilde \kappa^2-\kappa^2)^2 + \kappa^2 s^2} - \tilde \kappa^2 + \kappa^2.
	\end{align}
	Now, if we make use of the substitution $\kappa^2 + x^2 = \tilde \kappa^2 + y^2$ for the terms involving hyperbolic functions and the substitution $\kappa^2-y^2=\tilde \kappa^2-x^2$ for the terms involving trigonometric functions, we obtain the following:
	\begin{align*}
		\tilde{g}_{\pm}(z;\delta,l) &> \left[(\tilde \kappa^2+2y^2)x + 2\kappa y\right] lx - \left|(\tilde \kappa^2-2x^2)y -2\kappa x\right| ly \\
		&\ge l \left( x^2(\tilde \kappa^2+2y^2) - y^2|\tilde \kappa^2-2x^2| \right) \\
		&= \left\lbrace \begin{array}{l}
			l \tilde \kappa^2 (x^2+y^2) \\
			l \left( 4x^2y^2 + \tilde \kappa^2(\tilde \kappa^2-\kappa^2) \right)
		\end{array} \right. &
		\begin{array}{l}
			\text{if } \tilde \kappa^2 \le 2x^2, \\
			\text{if } \tilde \kappa^2 > 2x^2,
		\end{array} \\
		&= \left\lbrace \begin{array}{l}
			l \tilde \kappa^2 \sqrt{(\tilde \kappa^2-\kappa^2)^2 + \kappa^2 s^2} \\
			l \left( \tilde \kappa^4 + \kappa^2(s^2-\tilde \kappa^2) \right)
		\end{array} \right. &
		\begin{array}{l}
			\text{if } \tilde \kappa^2 \le 2x^2, \\
			\text{if } \tilde \kappa^2 > 2x^2,
		\end{array}
	\end{align*}
	and
	\begin{align*}
		\begin{split}
			\tilde{g}(z;\delta,l) &> l(l+2) \left( x^2(\tilde \kappa^2+2y^2)^2 - y^2(\tilde \kappa^2-2x^2)^2 \right) \\
			& \mathrel{\phantom{=}} \mathrel+ 4l\kappa xy \left( \tilde \kappa^2+2y^2 - |\tilde \kappa^2-2x^2| \right)
		\end{split} \\
		\begin{split}
			&= l(l+2) \left( \tilde \kappa^4 (x^2-y^2) + 4x^2y^2(2\tilde \kappa^2+y^2-x^2)\right) \\
			& \mathrel{\phantom{=}} \mathrel+ 4l\kappa xy \left( \tilde \kappa^2+2y^2 - |\tilde \kappa^2-2x^2| \right)
		\end{split} \\
		&= \left\lbrace \begin{array}{l}
			l(l+2) \left( \tilde \kappa^4 (\tilde \kappa^2-\kappa^2) + 4x^2y^2(\tilde \kappa^2+\kappa^2)\right) \\
			\quad \mathrel+ 8l\kappa^3xy \\
			l(l+2) \left( \tilde \kappa^4 (\tilde \kappa^2-\kappa^2) + 4x^2y^2(\tilde \kappa^2+\kappa^2)\right) \\
			\quad \mathrel+ 8l\kappa xy(x^2+y^2)
		\end{array} \right. &
		\begin{array}{l}
			\text{if } \tilde \kappa^2 \le 2x^2, \\ \\
			\text{if } \tilde \kappa^2 > 2x^2,
		\end{array} \\
		&= \left\lbrace \begin{array}{l}
			l(l+2) \left( \tilde \kappa^6 + \kappa^4s^2 + \tilde \kappa^2\kappa^2(s^2-\tilde \kappa^2)\right) \\
			\quad \mathrel+ 4l\kappa^4s \\
			l(l+2) \left( \tilde \kappa^6 + \kappa^4s^2 + \tilde \kappa^2\kappa^2(s^2-\tilde \kappa^2)\right) \\
			\quad \mathrel+ 4l\kappa^2s\sqrt{(\tilde \kappa^2-\kappa^2)^2 + \kappa^2 s^2}
		\end{array} \right. &
		\begin{array}{l}
			\text{if } \tilde \kappa^2 \le 2x^2, \\ \\
			\text{if } \tilde \kappa^2 > 2x^2.
		\end{array}
	\end{align*}
	From the penultimate expression in each case we see that for evanescent modes $\tilde k > k$ (i.e., $\tilde \kappa > \kappa$) we always have $\tilde{g}_{\pm}(z;\delta,l) > 0$ and $\tilde{g}(z;\delta,l) > 0$. Furthermore, from the final expressions we see that all modes $\tilde k \le \sigma$ (i.e., $\tilde \kappa \le s$) also give the desired positivity. Thus we deduce that when $\sigma \ge k$ we have positivity for all modes $\tilde k$ and hence $R_{2d}^{\mathrm{H}} < 1$. We also remark that modes $\tilde k \le k$ which are relatively close to $k$ are identified as those giving the worst bounds, suggesting these are the most problematic modes for the algorithm.

	If $\sigma < k$ we may still have positivity of $\tilde{g}_{\pm}(z;\delta,l)$ and $\tilde{g}(z;\delta,l)$ for all modes so long as $x$ or $lx$ is large enough so that the hyperbolic term, which is always positive, is larger than the magnitude of the trigonometric term in both \eqref{eq:gpmgt} and \eqref{eq:g1dgt}. Using \eqref{eq:2x^2and2y^2} and converting back to the original variables we have that
	\begin{align}
		\label{eq:x}
		x &= 2\delta \sqrt{\frac{1}{2}\left(\sqrt{(k^2-\tilde k^2)^2+\sigma^2k^2} + \tilde k^2 - k^2\right)},
	\end{align}
	while $lx$ has an identical expression except with $2\delta$ replaced by $L$. Thus we see that, between the parameters $\sigma$, $\delta$ and $L$, so long as they are sufficiently large we will have $\tilde{g}_{\pm}(z;\delta,l) > 0$ and $\tilde{g}(z;\delta,l) > 0$ for all modes $\tilde k$ and thus $R_{2d}^{\mathrm{H}} < 1$ as desired.
\end{proof}

To verify these results, we compare numerically the spectral radius of the iteration matrix with the theoretical limit for different values of $\sigma$. We choose here $k = 30$, $L=1$, $\hat{L}=1$ and $\delta = L/10$. From \Cref{fig:s012d,fig:s12d} we see that, as predicted, the Schwarz algorithm is not convergent for all Fourier modes when $\sigma$ is small, but becomes convergent for $\sigma$ sufficiently large. In particular, we see in \Cref{fig:s12d} that the method can be convergent for $\sigma \ll k$. As expected from our theory, the algorithm always converges well for evanescent modes ($\tilde{k} > k$).

\begin{figure}
	\centering
	\hfill
	\includegraphics[width=0.4\textwidth,trim=10cm 0.8cm 10cm 1.6cm ,clip]{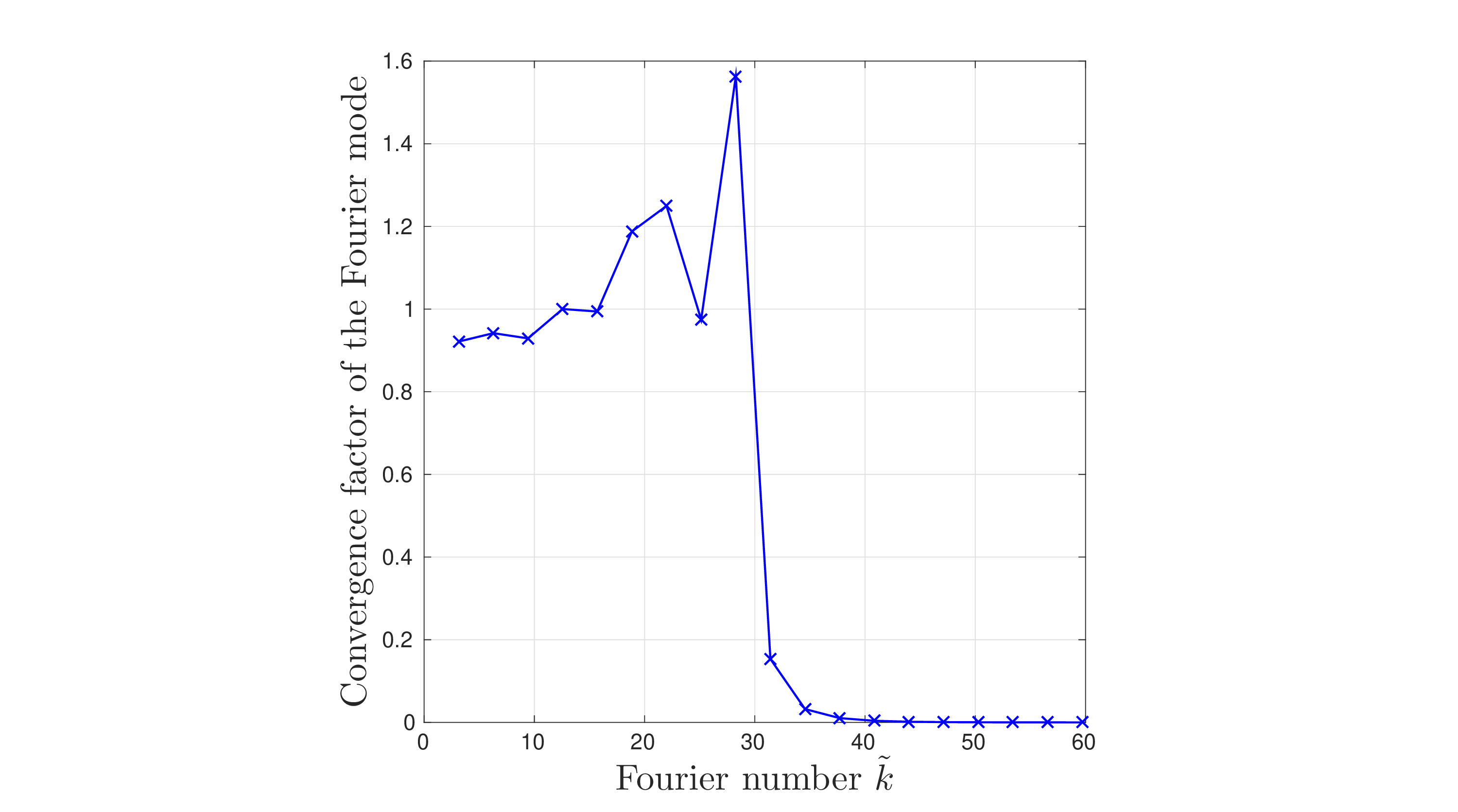}
	\hfill
	\includegraphics[width=0.4\textwidth,trim=10cm 0.8cm 10cm 1.6cm ,clip]{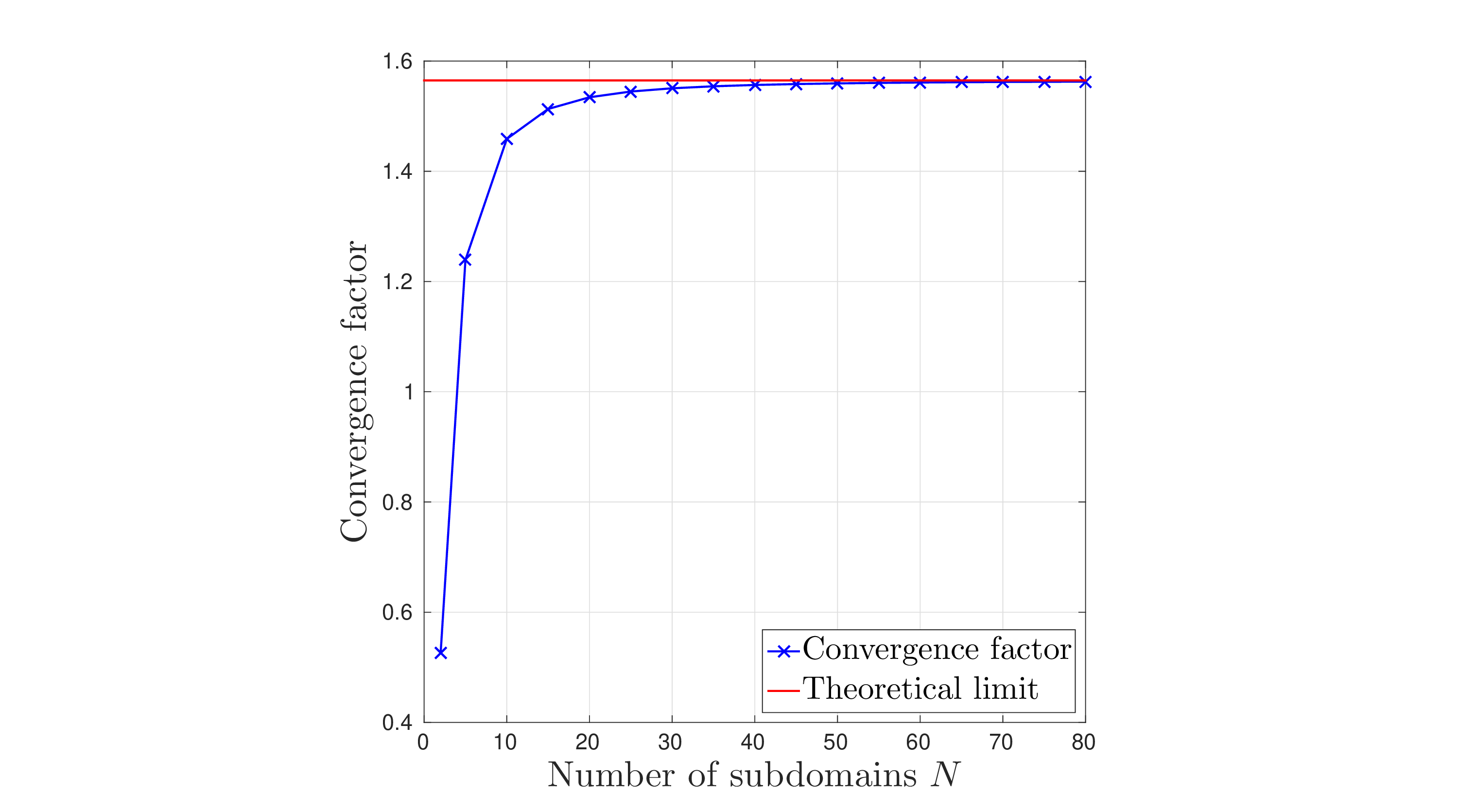}
	\hfill
	\caption{The convergence factor of each Fourier mode for $N=80$ (left) and the convergence factor of the full Schwarz algorithm for varying number of subdomains $N$ (right) when $\sigma = 0.1$.}
	\label{fig:s012d}
\end{figure}

\begin{figure}
	\centering
	\hfill
	\includegraphics[width=0.4\textwidth,trim=10cm 0.8cm 10cm 1.6cm ,clip]{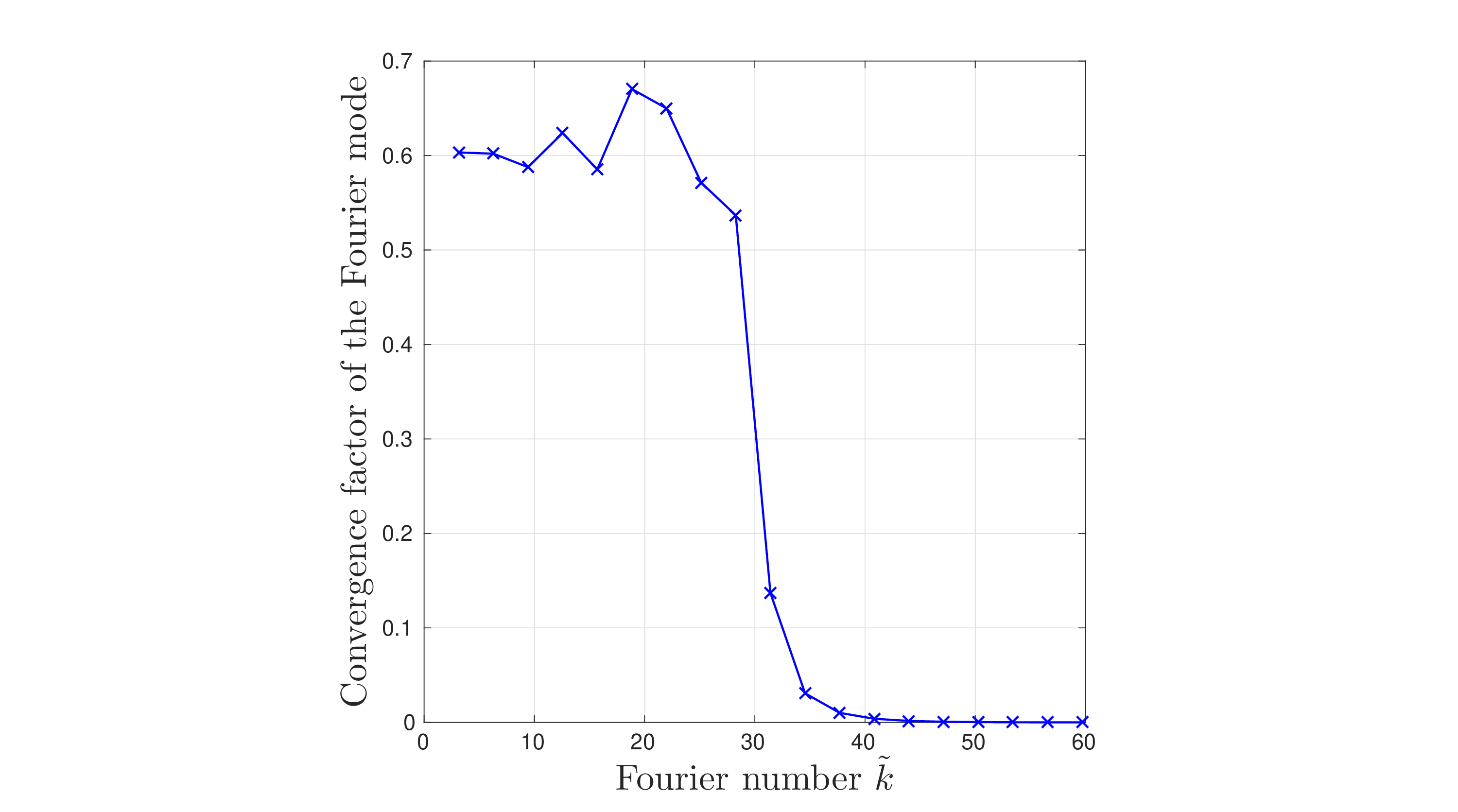}
	\hfill
	\includegraphics[width=0.4\textwidth,trim=10cm 0.8cm 10cm 1.6cm ,clip]{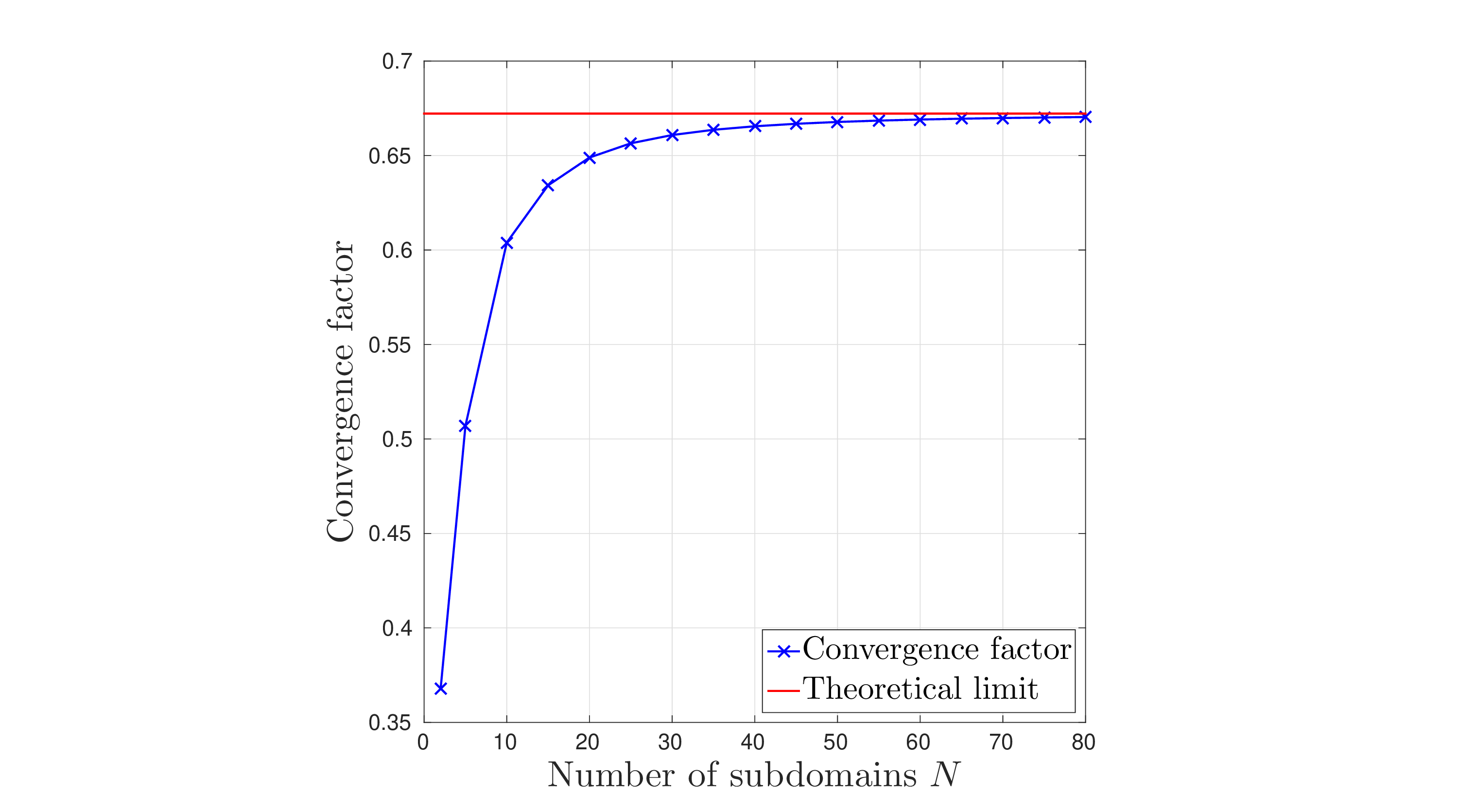}
	\hfill
	\caption{The convergence factor of each Fourier mode for $N=80$ (left) and the convergence factor of the full Schwarz algorithm for varying number of subdomains $N$ (right) when $\sigma = 1$.}
	\label{fig:s12d}
\end{figure}

Similarly to the one-dimensional case we can also consider the question of $k$-robustness:

\begin{corollary}[A case of $k$-independent convergence]
	\label{corollary:k-indpendenceHelmholtz2D}
	Suppose $\alpha = ik$ (the case of classical impedance conditions) and that $\sigma = \sigma_{0} k$ for some constant $\sigma_{0}$. Consider a $k$-dependent domain decomposition given by $L = L_{0} k^{-1}$ and $\delta = \delta_{0} k^{-1}$, that is the subdomain size and overlap shrink inversely proportional to the wave number. Then the convergence factor $R_{2d}^{\mathrm{H}}$ can be bounded above by a $k$-independent value and this bound becomes tight as $k\rightarrow\infty$. As such, the convergence of the corresponding Schwarz method is ultimately independent of the wave number $k$ as it increases. Under the additional assumptions of \Cref{theorem:2dhelmholtz} for convergence (now on $\sigma_{0}$, $L_{0}$ and $\delta_{0}$), we thus deduce that the approach will ultimately be $k$-robust and independent of the number of subdomains.
\end{corollary}
\begin{proof}
	The proof is similar to the one-dimensional case except that now we must consider the Fourier number $\tilde{k}$. To do so, we let $\tilde{k}^2 = \beta k^2$. In this scenario, the coefficients $a$ and $b$ of the iteration matrix depend on $k$ only through $\beta$. However, in the final convergence factor $R_{2d}^{\mathrm{H}}$ we take the supremum over all $\tilde{k}$, namely now over a discrete set of positive $\beta$ values. This is bounded above by the supremum over all $\beta \in \mathbb{R}^{+}$, which is then independent of $k$, the supremum being finite since the bounds derived in \Cref{theorem:2dhelmholtz} do not rely on the discrete nature of $\tilde{k}$ and so can be readily applied, translated into $\beta$. Note that as $k\rightarrow\infty$ the discrete set of $\beta$ values becomes dense in $\mathbb{R}^{+}$ so this supremum bound becomes tight. Thus we will ultimately have $k$-robustness. Combining with \Cref{theorem:2dhelmholtz} we further obtain that ultimately the convergence will also be independent of the number of subdomains.
\end{proof}

\begin{remark}
	\label{remark:tildek-supremumHelmholtz2D}
	We note an empirical observation that, for reasonable values of $\sigma$, $\delta$ and $L$ (namely when these parameters are not too small, essentially the same conditions required for convergence, but also neither of $\delta$ or $\sigma$ being too large), the value of $\tilde{k}$ giving the supremum of $R_{1d}^{\mathrm{H}}(\tilde k)$ lies in a small neighbourhood around $k$ (equivalent to $\beta = 1$ in the above proof). This is consistent with other works in the literature, e.g., \cite{Gander:2002:OSM,Conen:2015}, where the most problematic modes are those close to the cut-off $k$. In this case, a series expansion around $\tilde{k} = k$ shows that $k \delta$ and $k L$ being fixed are the requirements on the domain decomposition parameters in order for the algorithm to be $k$-independent; see the supplementary \texttt{Maple} worksheets.
\end{remark}

For more general theory on $k$-robustness of the one-level method and rigorous GMRES bounds, see \cite{Graham:2020:DDI}. As in the one-dimensional case, we can link $k$ and $N$ so that we consider solving on a fixed domain a family of problems with increasing wave number using an increasing number of subdomains and, under the conditions of \Cref{theorem:2dhelmholtz} and  \Cref{corollary:k-indpendenceHelmholtz2D}, our theory shows that the Schwarz algorithm will ultimately be $k$-robust and weakly scalable.

\begin{remark}
	We have focused here on the case of an overlapping domain decomposition. While the algorithm can also work in the non-overlapping case, it typically has a very poor behaviour. It is known from the literature (for example by setting the parameters to zero in formula (3.2) from \cite{Gander:2002:OSM}) that if $\sigma=0$ in the case of a decomposition into two subdomains, the purely iterative algorithm does not converge for evanescent modes ($\tilde k >k$), the convergence factor being equal to $1$. By increasing $\sigma$, the convergence factor can be lowered but only a little (it remains close to one) and the algorithm continues to have very poor convergence properties for evanescent modes. This is illustrated in \Cref{fig:rhonovr}, where we take the same parameter values as in our previous results ($k = 30$, $L=1$ and $\hat{L}=1$), and can be proven by similar techniques to those used in the overlapping case.

	\begin{figure}
		\centering
		\hfill
		\includegraphics[width=0.4\textwidth,trim=10cm 0.8cm 10cm 1.6cm ,clip]{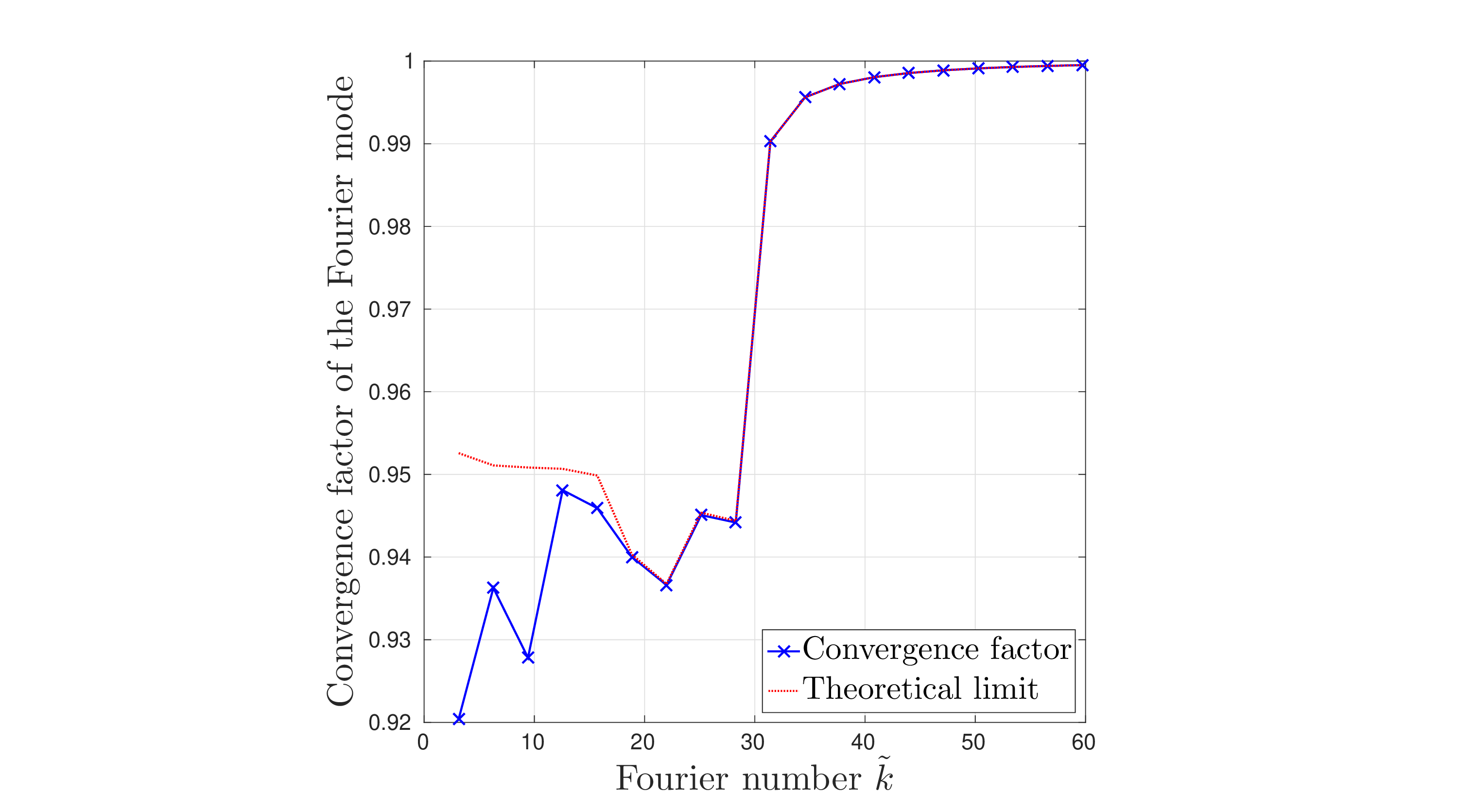}
		\hfill
		\includegraphics[width=0.4\textwidth,trim=10cm 0.8cm 10cm 1.6cm ,clip]{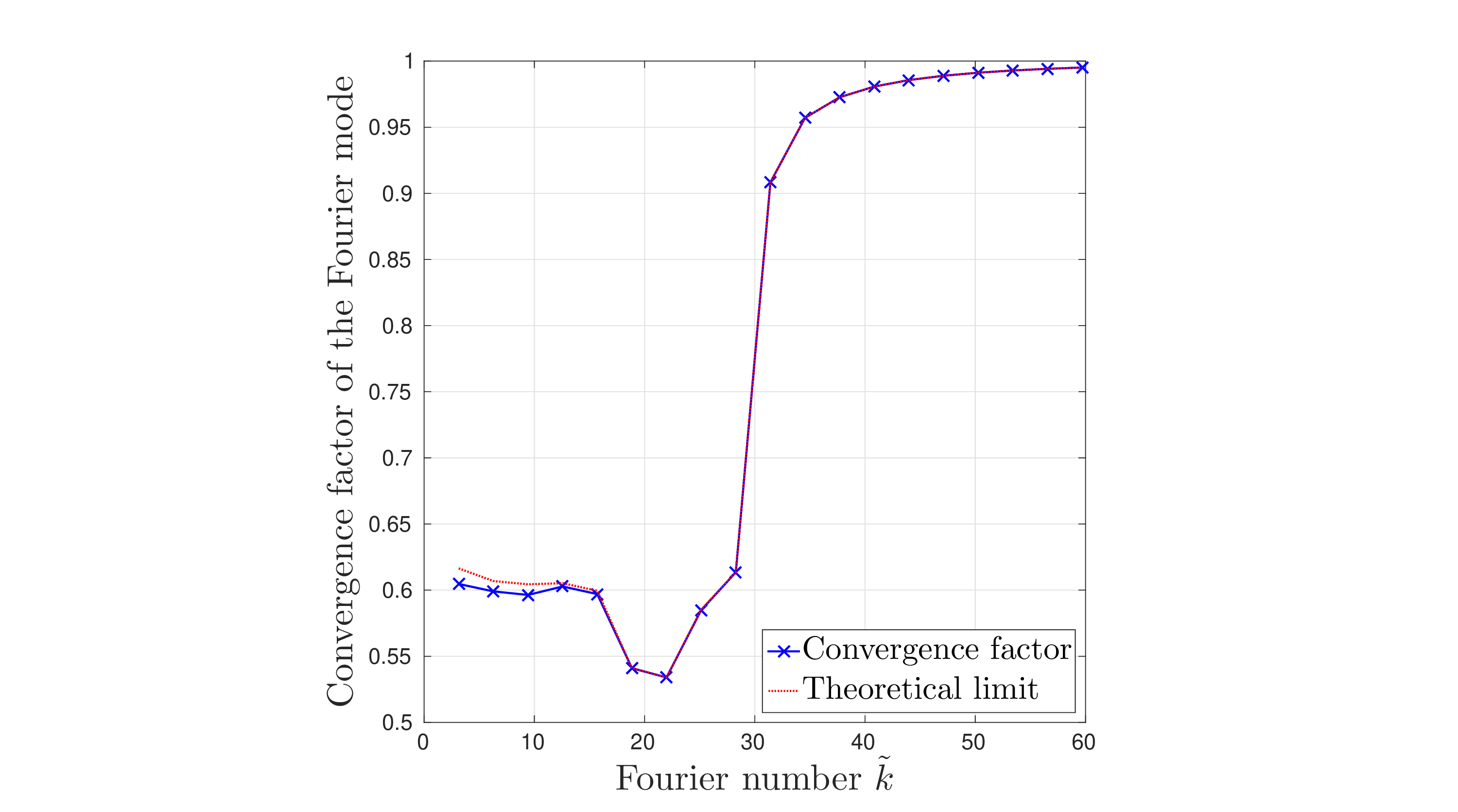}
		\hfill
		\caption{The convergence factor of each Fourier mode in the non-overlapping case for $N=80$ and $\sigma=0.1$ (left) and $\sigma=1$ (right).}
		\label{fig:rhonovr}
	\end{figure}
\end{remark}

We also note a fundamental difference between the one-dimensional and two-dimensional cases from the scalability point of view. Whereas in the first case independence to the number of subdomains is achieved simply by taking $\sigma>0$, in the two-dimensional case things become more complex. This is consistent with previous convergence studies, starting from that in the seminal work on optimised transmission conditions \cite{Gander:2002:OSM}, where it has been observed that propagative and evanescent modes behave differently and the iterative algorithm does not converge for the cut-off frequency $k$. The maximum of the convergence factor is usually attained in a neighbourhood of $\tilde{k}= k$ and can be made sufficiently small when $\sigma$ is taken large enough; in this case we can achieve scalability and $k$-robustness. We note that this kind of discrepancy, between one- and two-dimensional problems, is typical for the Helmholtz equation and cannot be observed in the case of the Laplace equation.

\subsection{The transverse electric Maxwell's equations}
\label{sec:2dmaxwell}
We now apply the same ideas to the transverse electric Maxwell's equations with damping in the frequency domain. For an electric field $\mathbf{E}=(E_x,E_y)$, these equations are expressed as
\begin{align}\label{eq:2d0}
	\begin{split}
		& {\cal L}\mathbf{E} \vcentcolon= -k^2\mathbf{E}+ \nabla \times (\nabla \times \mathbf{E}) +ik \sigma \mathbf{E} = \mathbf{0} \\
		& \Leftrightarrow \left\{\begin{array}{r@{}l}
			- k^2 E_x - \partial_{yy} E_x + \partial_{xy} E_y + ik\sigma E_x &{}= 0, \\
			- k^2 E_y - \partial_{xx} E_y + \partial_{xy} E_x + ik\sigma E_y &{}= 0,
		\end{array}\right.
	\end{split}
\end{align}
for $(x,y) \in \Omega$. The boundary conditions on the top and bottom boundaries ($y=0$ and $y=\hat L$) are perfect electric conductor (PEC) conditions, the equivalent of Dirichlet conditions for Maxwell's equations:
\begin{align}\label{eq:2dPEC}
	\mathbf{E} \times \mathbf{n} = \mathbf{0} \Leftrightarrow E_x &= 0, & y &= \{0, \hat L\}.
\end{align}
On the left and right boundaries ($x=a_1$ and $x=b_N$) we use impedance boundary conditions\footnote{Note that in rewriting the impedance conditions we can use the three-dimensional definition of the operators, i.e., $\mathbf{E}=(E_x,E_y,0)$ and $\mathbf{n}=(1,0,0)$ for the right boundary and $\mathbf{n}=(-1,0,0)$ for the left boundary.}:
\begin{align}\label{eq:2dimp}
	\begin{split}
		& (\nabla \times \mathbf{E} \times \mathbf{n}) \times \mathbf{n} + ik \mathbf{E} \times \mathbf{n} = \mathbf{g} \\
		& \Leftrightarrow \left\{\begin{array}{r@{}ll}
			{\cal B}_l\mathbf{E} \vcentcolon= (-\partial_x + ik) E_y + \partial_y E_x &{}= g_1, & x = a_1,\\
			{\cal B}_r\mathbf{E} \vcentcolon= (\partial_x + ik) E_y - \partial_y E_x &{}= - g_2, & x = b_N.
		\end{array}\right.
	\end{split}
\end{align}
The same conditions will be used at the interfaces between subdomains, akin to the classical algorithm defined in \cite{Despres:1992:DDM}. The Maxwell problem \eqref{eq:2d0}--\eqref{eq:2dimp} constitutes a {\it ``two-dimensional wave-guide''} model.

Let us denote by $\mathbf{E}_j^n$ the approximation to the solution in subdomain $j$ at iteration $n$. Starting from an initial guess $\mathbf{E}_j^0$, we compute $\mathbf{E}_j^n$ from the previous values $\mathbf{E}_j^{n-1}$ by solving the following local boundary value problems
\begin{align}\label{eq:sch2d}
	\left\{\begin{array}{r@{}ll}
		{\cal L} \mathbf{E}_j^n &{}= \mathbf{0}, & x \in \Omega_j,\\
		{\cal B}_l\mathbf{E}_j^n &{}= {\cal B}_l \mathbf{E}_{j-1}^{n-1}, & x = a_j,\\
		{\cal B}_r\mathbf{E}_j^n &{}= {\cal B}_r \mathbf{E}_{j+1}^{n-1}, & x = b_j,\\
		E_{x,j}^n &{}= 0, & y \in \{0,\hat L\},
	\end{array}\right.
\end{align}
for the interior subdomains ($1 < j < N$), while for the first ($j=1$) and last ($j=N$) subdomain we impose ${\cal B}_l \mathbf{E}_1^n {}= g_1$ when $x = a_1$ and ${\cal B}_r \mathbf{E}_N^n {}= -g_2$ when $x = b_N$. To study the convergence of the Schwarz algorithm we define the local error in each subdomain $j$ at iteration $n$ as $\mathbf{e}^n_j = \mathbf{E}|_{\Omega_j} - \mathbf{E}^n_j$. Note that these errors verify boundary value problems which are the homogeneous counterparts of \eqref{eq:sch2d}.

Due to the PEC boundary conditions on the top and bottom boundaries of each rectangular subdomain we can use the following Fourier series ansatzes to compute the local solutions of ${\cal L}\mathbf{e}^n_j = 0$:
\begin{align}
	\label{eq:Fourier2}
	e_{x,j}^n &= \sum_{m=1}^{\infty} v^n_j(x,\tilde k) \sin(\tilde k y), & e_{y,j}^n &= \sum_{m=1}^{\infty} w^n_j(x,\tilde k) \cos(\tilde k y), & \tilde k &= \frac{m\pi}{\hat L}, \ m \in \mathbb{N}.
\end{align}
By plugging the expressions for $e_{x,j}^n$ and $e_{y,j}^n$ into ${\cal L}\mathbf{e}^n_j = \mathbf{0}$, a simple computation shows that, for each Fourier number $\tilde k$, we have the general solutions
\begin{align}
	\label{eq:vjwj}
	v^n_j(x,\tilde k) &= -\alpha^n_j \frac{\tilde k}{{\zeta}} e^{-\zeta x} + \beta^n_j \frac{\tilde k}{{\zeta}} e^{\zeta x}, & w^n_j(x,\tilde k) &= \alpha^n_j e^{-\zeta x} + \beta^n_j e^{\zeta x},
\end{align}
where $\zeta(\tilde k) = \sqrt{ik\sigma+\tilde k^2-k^2}$. From these formulae we can see easily that
\begin{align}
	\label{eq:vwrel}
	\partial_x v^n_j &= \tilde k w^n_j, & \partial_x w^n_j &= \frac{\zeta^2}{\tilde k}v^n_j.
\end{align}
In order to benefit again from the analysis in the one-dimensional case, we first prove the following result.
\begin{lemma}[Maxwell reduction]
	For each Fourier number $\tilde k$, we have that both $v^n_j(x,\tilde k)$ and $w^n_j(x,\tilde k)$ are solutions of the following one-dimensional problem:
	\begin{align}
		\label{eq:2dmax}
		\left\{ \begin{array}{r@{}ll}
			(ik\sigma+\tilde k^2 -k^2) u_j^{n} - \partial_{xx}u_j^{n} &{}= 0, & x \in (a_j,b_j), \\
			{\cal B}_{l,\sigma} u_j^{n}(x,\tilde k) &{}= {\cal B}_{l,\sigma} u_{j-1}^{n-1}(x,\tilde k), & x = a_j, \\
			{\cal B}_{r,\sigma} u_j^{n}(x,\tilde k) &{}= {\cal B}_{r,\sigma} u_{j+1}^{n-1}(x,\tilde k), & x = b_j,
		\end{array} \right.
	\end{align}
	where ${\cal B}_{l,\sigma} = -\partial_x + ik + \sigma$ and ${\cal B}_{r,\sigma} = \partial_x + ik + \sigma$.
\end{lemma}
\begin{proof}
	Let us notice first that, because of \eqref{eq:vwrel}, we have
	\begin{align*}
		\partial_x e_{x,j}^n+ \partial_y e_{y,j}^n = \sum_{m=1}^{\infty} \left(\partial_x v_{j}^n - \tilde k w_{j}^n\right)\sin(\tilde k y) = 0.
	\end{align*}
	If we use this in the error equation ${\cal L} \mathbf{e}^n_j = \mathbf{0}$ we obtain that both $v^n_j(x,\tilde k)$ and $w^n_j(x,\tilde k)$ satisfy, for each $\tilde k$, the one-dimensional equation $(ik\sigma+\tilde k^2-k^2) u_j^{n} - \partial_{xx} u_j^{n} = 0$. Let us analyse now the boundary conditions. With the help of \eqref{eq:vwrel}, we consider the right boundary and note that the left one can be treated similarly:
	\begin{align*}
		{\cal B}_r \mathbf{e}^n_j &= (\partial_x + ik) e_{y,j}^n -\partial_y e_{x,j}^n = \sum_{m=1}^{\infty} ( (\partial_x +ik)w_{j}^n -\tilde k v_{j}^n) \cos(\tilde k y) \\
		&= \sum_{m=1}^{\infty} \left ( \frac{ik}{\tilde k} \partial_x v_{j}^n + \left(\frac{\zeta^2}{\tilde k} -\tilde k\right ) v_{j}^n\right) \cos(\tilde k y) = \sum_{m=1}^{\infty} \frac{ik}{\tilde k} {\cal B}_{r,\sigma} v_{j}^n \cos(\tilde k y).
	\end{align*}
	Thus, imposing transfer of boundary data with ${\cal B}_r \mathbf{e}^n_j$ is equivalent to that with ${\cal B}_{r,\sigma} v^n_j$, for each Fourier number $\tilde k$.
\end{proof}

It is now clear that the analysis of the two-dimensional case can again be derived from the one-dimensional case. That is, the result from \Cref{lemma:basic1d} applies here if we replace $\alpha$ with $ik+\sigma$ and with $\zeta$ being defined by \eqref{eq:lamk}. Let us denote the resulting iteration matrix, for each $\tilde k$, by ${\cal T}_{1d}^{\mathrm{M}}(\tilde k)$ and let $R_{1d}^{\mathrm{M}}(\tilde k) \vcentcolon= \lim_{N\rightarrow\infty} \rho({\cal T}_{1d}^{\mathrm{M}}(\tilde k))$ with $R_{2d}^{\mathrm{M}} = \sup_{\tilde k} R_{1d}^{\mathrm{M}}(\tilde k)$. We can now state our main convergence result for the two-dimensional Maxwell problem.

\begin{theorem}[Convergence of the Schwarz algorithm for Maxwell in 2D]
	\label{theorem:2dmaxwell}
	For all $k>0$, $\sigma>0$, $\delta>0$ and $L>0$ we have that $R_{1d}^{\mathrm{M}}(\tilde k) < 1$ for all evanescent modes $\tilde k > k$. Furthermore, under the assumption that between them $\sigma$, $\delta$ and $L$ are sufficiently large we have that $R_{2d}^{\mathrm{M}} < 1$. In particular, this is true when $\sigma \ge k$ for all $\delta>0$ and $L>0$. Therefore the convergence will ultimately be independent of the number of subdomains (we say that the Schwarz method will scale).
\end{theorem}
\pagebreak
\begin{proof}
	By \Cref{lemma:LimitingSpectralRadius1D} we see that it is enough to study the sign of $g_{\pm}(z;\delta,l)$ and $g(z;\delta,l)$. To assist, we use the scaled notation $\kappa = 2\delta k$, $\tilde \kappa = 2\delta \tilde k$ and $s = 2\delta\sigma$ akin to \eqref{eq:params}. We can see that if $\alpha = ik +\sigma$ then for $z\vcentcolon= x+iy$ \eqref{eq:v} becomes
	\begin{align*}
		\Re v &= \frac{-\kappa^2 - s^2 + x^2 + y^2}{(\kappa+y)^2 + (s+x)^2}, &
		\Im v &= \frac{2sy - 2\kappa x}{(\kappa+y)^2 + (s+x)^2}, &
		|v|^2 &= \frac{(\kappa-y)^2 + (s-x)^2}{(\kappa+y)^2 + (s+x)^2},
	\end{align*}
	where $|v|^2 < 1$. We can now simplify $g_{\pm}(z;\delta,l)$ in \eqref{eq:g} using these formulae to give
	\begin{subequations}
		\label{eq:gpm2d}
		\begin{align}
			\label{eq:gpm2dg}
			g_{\pm}(z;\delta,l) &= \frac{4e^{x(l+1)}}{(\kappa+y)^2 + (s+x)^2} \tilde{g}_{\pm}(z;\delta,l) \\
			\begin{split}
				\label{eq:gpm2dgt}
				\tilde{g}_{\pm}(z;\delta,l) &= [(\kappa^2+s^2+x^2+y^2) \sinh(x) + 2(\kappa y+sx) \cosh(x)]\sinh(lx) \\
				& \mathrel{\phantom{=}} \mathrel\pm [(\kappa^2+s^2-x^2-y^2) \sin(y) + 2(sy-\kappa x) \cos(y)]\sin(ly).
			\end{split}
		\end{align}
	\end{subequations}
	Proceeding as before, using $\tilde \kappa^2 - \kappa^2 = x^2 - y^2$ and the bounds \eqref{eq:HypTrigBounds}, we derive that
	\begin{align*}
		\tilde{g}_{\pm}(z;\delta,l) &> l (\kappa^2+s^2+2s+x^2+y^2) (\tilde \kappa^2-\kappa^2)
	\end{align*}
	which is positive for all evanescent modes $\tilde k > k$. Similarly, simplifying $g(z;\delta,l)$ in \eqref{eq:ga} we find that
	\begin{subequations}
		\label{eq:g2d}
		\begin{align}
			\label{eq:g2dg}
			g(z;\delta,l) &= \frac{4e^{2x(l+1)}}{((\kappa+y)^2 + (s+x)^2)^2} \tilde{g}(z;\delta,l) \\
			\begin{split}
				\label{eq:g2dgt}
				\tilde{g}(z;\delta,l) &= \bigl[ ((\kappa^2+s^2+x^2+y^2)^2+4(\kappa y+sx)^2) \sinh(x(l+2)) \\
				& \qquad\qquad \mathrel+ 4(\kappa y+sx)(\kappa^2+s^2+x^2+y^2) \cosh(x(l+2)) \bigr] \sinh(lx) \\
				& \mathrel{\phantom{=}} \mathrel+ \bigl[ ((-\kappa^2-s^2+x^2+y^2)^2-4(\kappa x-sy)^2) \sin(y(l+2)) \\
				& \qquad\qquad \mathrel+ 4(\kappa x-sy)(-\kappa^2-s^2+x^2+y^2) \cos(y(l+2)) \bigr] \sin(ly),
			\end{split}
		\end{align}
	\end{subequations}
	from which we can obtain the bound
	\begin{align*}
		\begin{split}
			\tilde{g}(z;\delta,l) &> l(l+2) \left( (\kappa^2+s^2+x^2+y^2)^2 + 4s^2(x^2+y^2) + 8\kappa sxy \right) (\tilde \kappa^2-\kappa^2) \\
			& \mathrel{\phantom{=}} \mathrel+ 4ls \left( \kappa^2+s^2+x^2+y^2 \right) (\tilde \kappa^2-\kappa^2).
		\end{split}
	\end{align*}
	Again, this is positive for all evanescent modes and thus we deduce that $R_{1d}^{\mathrm{M}}(\tilde k) < 1$ for all $\tilde k > k$.

	We now refine these bounds, as in the proof of \Cref{theorem:2dhelmholtz} and using the same identities and substitutions. For $g_{\pm}(z;\delta,l)$ we first obtain
	\begin{align*}
		\tilde{g}_{\pm}(z;\delta,l) &> l \left( x^2(s^2+\tilde \kappa^2+2y^2) - y^2\left|s^2+\tilde \kappa^2-2x^2\right| + 2x(\kappa y+sx) - 2y\left|\kappa x-sy\right| \right),
	\end{align*}
	and split into four cases based on the sign of each term we take the absolute value of. Consider first the case $s^2 + \tilde \kappa^2 \le 2x^2$ and $\kappa x \le sy$, then
	\begin{align*}
		\tilde{g}_{\pm}(z;\delta,l) &> l \left( (s^2+\tilde \kappa^2)(x^2+y^2) + 4\kappa xy + 2s(x^2-y^2) \right) \\
		&= l \left( (s^2+\tilde \kappa^2)\sqrt{(\tilde \kappa^2-\kappa^2)^2 + \kappa^2 s^2} + 2\tilde \kappa^2s \right).
	\end{align*}
	Now consider the case $s^2 + \tilde \kappa^2 > 2x^2$ and $\kappa x > sy$ where we find that
	\begin{align*}
		\tilde{g}_{\pm}(z;\delta,l) &> l \left( 4x^2y^2 + (s^2+\tilde \kappa^2)(x^2-y^2) + 2s(x^2+y^2) \right) \\
		&= l \left( \tilde \kappa^2(s^2+\tilde \kappa^2-\kappa^2) + 2s\sqrt{(\tilde \kappa^2-\kappa^2)^2 + \kappa^2 s^2} \right).
	\end{align*}
	The remaining cases follow as combinations of the previous two cases and we deduce, in the case $s^2 + \tilde \kappa^2 \le 2x^2$ and $\kappa x > sy$, that
	\begin{align*}
		\tilde{g}_{\pm}(z;\delta,l) &> l (s^2+\tilde \kappa^2+2s)\sqrt{(\tilde \kappa^2-\kappa^2)^2 + \kappa^2 s^2},
	\end{align*}
	while the case $s^2 + \tilde \kappa^2 > 2x^2$ and $\kappa x \le sy$ gives
	\begin{align*}
		\tilde{g}_{\pm}(z;\delta,l) &> l \tilde \kappa^2(s^2+\tilde \kappa^2-\kappa^2+2s).
	\end{align*}
	Turning to $\tilde{g}(z;\delta,l)$, we first derive that
	\begin{align*}
		\begin{split}
			\tilde{g}(z;\delta,l) &> l(l+2) \bigl[ x^2((s^2+\tilde \kappa^2+2y^2)^2+4(\kappa y+sx)^2) \\
			& \mathrel{\phantom{=}} \mathrel- y^2((s^2+\tilde \kappa^2-2x^2)^2+4(\kappa x-sy)^2) \bigr] \\
			& \mathrel{\phantom{=}} \mathrel+ 4l \left( x(\kappa y+sx)(s^2+\tilde \kappa^2+2y^2) - y\left| (\kappa x-sy)(s^2+\tilde \kappa^2-2x^2) \right| \right),
		\end{split}
	\end{align*}
	from which we see that we need to analyse just two sets of combined cases. First consider when both $s^2 + \tilde \kappa^2 \le 2x^2$ and $\kappa x \le sy$ or both $s^2 + \tilde \kappa^2 > 2x^2$ and $\kappa x > sy$, yielding
	\begin{align*}
		\begin{split}
			\tilde{g}(z;\delta,l) &> l(l+2) \bigl[ (s^2+\tilde \kappa^2)^2(x^2-y^2) + 4x^2y^2(2s^2+2\tilde \kappa^2+y^2-x^2) \\
			& \mathrel{\phantom{=}} \mathrel+ 4s(x^2+y^2)(2\kappa xy + s(x^2-y^2)) \bigr] + 4l (x^2+y^2)(2\kappa xy + s(s^2+\tilde \kappa^2))
		\end{split} \\
		\begin{split}
			&= l(l+2) \Bigl[ \kappa^2s^2(s^2+\kappa^2) + \tilde \kappa^2(s^2+\tilde \kappa^2)(s^2+\tilde \kappa^2-\kappa^2) \\
			& \mathrel{\phantom{=}} \mathrel+ 4\tilde \kappa^2s^2\sqrt{(\tilde \kappa^2-\kappa^2)^2 + \kappa^2 s^2} \Bigr] + 4ls (s^2+\tilde \kappa^2+\kappa^2)\sqrt{(\tilde \kappa^2-\kappa^2)^2 + \kappa^2 s^2}.
		\end{split}
	\end{align*}
	On the other hand, in the second set of cases when both $s^2 + \tilde \kappa^2 \le 2x^2$ and $\kappa x > sy$ or both $s^2 + \tilde \kappa^2 > 2x^2$ and $\kappa x \le sy$ we have
	\begin{align*}
		\begin{split}
			\tilde{g}(z;\delta,l) &> l(l+2) \bigl[ (s^2+\tilde \kappa^2)^2(x^2-y^2) + 4x^2y^2(2s^2+2\tilde \kappa^2+y^2-x^2) \\
			& \mathrel{\phantom{=}} \mathrel+ 4s(x^2+y^2)(2\kappa xy + s(x^2-y^2)) \bigr] \\
			& \mathrel{\phantom{=}} \mathrel+ 4l \left( 2\kappa xy(s^2+\tilde \kappa^2+y^2-x^2) + s(4x^2y^2+(s^2+\tilde \kappa^2)(x^2-y^2)) \right)
		\end{split} \\
		\begin{split}
			&= l(l+2) \Bigl[ \kappa^2s^2(s^2+\kappa^2) + \tilde \kappa^2(s^2+\tilde \kappa^2)(s^2+\tilde \kappa^2-\kappa^2) \\
			& \mathrel{\phantom{=}} \mathrel+ 4\tilde \kappa^2s^2\sqrt{(\tilde \kappa^2-\kappa^2)^2 + \kappa^2 s^2} \Bigr] + 4ls \left( \kappa^2 (s^2+\kappa^2) + \tilde \kappa^2(s^2+\tilde \kappa^2-\kappa^2) \right).
		\end{split}
	\end{align*}
	Summarising, we see that all cases give $\tilde{g}_{\pm}(z;\delta,l) > 0$ and $\tilde{g}(z;\delta,l) > 0$ for all modes $\tilde k$ satisfying $\tilde k^2 \ge k^2 - \sigma^2$ (i.e., $\tilde \kappa^2 \ge \kappa^2 - s^2$). From this we can deduce that when $\sigma \ge k$ we have positivity for all modes $\tilde k$ and hence $R_{2d}^{\mathrm{M}} < 1$. Note that $\sigma \ge k$ is far from a necessary requirement and it is clear that there is some slack in these bounds. We also remark from this analysis that modes $\tilde k \le \sqrt{k^2-\sigma^2}$ which are relatively close to $\sqrt{k^2-\sigma^2}$ yield the poorest bounds, suggesting they are the most problematic for the algorithm. Indeed, we may have $R_{1d}^{\mathrm{M}}(\tilde k) \ge 1$ when $\tilde k \le \sqrt{k^2-\sigma^2}$ for some choices of problem parameters. However, as in \Cref{theorem:2dhelmholtz} we can force positivity of $\tilde{g}_{\pm}(z;\delta,l)$ and $\tilde{g}(z;\delta,l)$ for all modes so long as $x$ or $lx$ is large enough. Since $x$ and $lx$ take the same expressions as in \Cref{theorem:2dhelmholtz} we can similarly deduce that, so long as the parameters $\sigma$, $\delta$ and $L$ between them are sufficiently large, we will have $\tilde{g}_{\pm}(z;\delta,l) > 0$ and $\tilde{g}(z;\delta,l) > 0$ for all modes $\tilde k$ and thus the required conclusion that $R_{2d}^{\mathrm{M}} < 1$.
\end{proof}
\pagebreak
\begin{corollary}[A case of $k$-independent convergence]
	\label{corollary:k-indpendenceMaxwell2D}
	Suppose that $\sigma = \sigma_{0} k$ for some constant $\sigma_{0}$. Consider a $k$-dependent domain decomposition given by $L = L_{0} k^{-1}$ and $\delta = \delta_{0} k^{-1}$, that is the subdomain size and overlap shrink inversely proportional to the wave number. Then the convergence factor $R_{2d}^{\mathrm{M}}$ can be bounded above by a $k$-independent value and this bound becomes tight as $k\rightarrow\infty$. As such, the convergence of the corresponding Schwarz method is ultimately independent of the wave number $k$ as it increases. Under the additional assumptions of \Cref{theorem:2dmaxwell} for convergence (now on $\sigma_{0}$, $L_{0}$ and $\delta_{0}$), we thus deduce that the approach will ultimately be $k$-robust and independent of the number of subdomains.
\end{corollary}

The proof is identical to that of \Cref{corollary:k-indpendenceHelmholtz2D} and we find similar empirical observations to those in \Cref{remark:tildek-supremumHelmholtz2D}. As before, we can also link $k$ and $N$ so that we consider solving on a fixed domain a family of problems with increasing wave number using an increasing number of subdomains and, under the conditions of \Cref{theorem:2dmaxwell} and \Cref{corollary:k-indpendenceMaxwell2D}, our theory shows that the algorithm will ultimately be $k$-robust and weakly scalable. Note that in the Maxwell case we are not aware of any theory showing $k$-robustness of the one-level method.

\section{Numerical simulations on the discretised equation}
\label{sec:num}

Although extensive numerical results are beyond the scope of this paper, in the following short section we will show some simulations which confirm our theory within the more practical setting of using an iterative Krylov method to accelerate convergence, with the Schwarz method being used as a preconditioner.

We focus here on the two-dimensional Helmholtz equation, as described in \Cref{sec:2dcase} (with $\alpha = ik$), where a (horizontal) plane wave is incoming from the left boundary and homogeneous Dirichlet boundary conditions are imposed on the top and bottom boundaries, giving a wave-guide problem. A second test case we consider is the propagation of such a wave in free space (i.e., when impedance boundary conditions are imposed on the whole boundary). While not covered by our theory, we will nonetheless observe similar conclusions, illustrating that the results apply more widely than within the restrictions of our theoretical assumptions. In our simulations, each subdomain is a unit square split uniformly with a fixed number of grid points in each direction. New subdomains are added on the right so that, with $N$ subdomains, the whole domain is $\Omega = (0,N)\times (0,1)$.

To discretise we use a uniform square grid in each direction and triangulate to form P1 elements. As we increase $k$ we increase the number of grid points proportional to $k^{3/2}$ in order to ameliorate the pollution effect \cite{Babuska:1997:IPE}. We use an overlap of size $2h$, with $h$ being the mesh size. All computations are performed using FreeFem (\url{http://freefem.org/}), in particular using the \texttt{ffddm} framework. We solve the discretised problem using GMRES where the parallel Schwarz method with Robin conditions is used as a preconditioner. In particular, we use right-preconditioned GMRES and terminate when a relative residual tolerance of $10^{-6}$ is reached. The construction of the domain decomposition preconditioner is described in detail in \cite{Bonazzoli:2019:ADD,Dolean:2020:IFD}. The preconditioner, which arises naturally as the discretised version of the parallel Schwarz method with Robin conditions we have studied (see, e.g., \cite{StCyr:2007:OMA}), is known as the one-level optimised restricted additive Schwarz (ORAS) preconditioner. This ORAS preconditioner is given by
\begin{align*}
	\mathbf{M}^{-1} = \sum_{i=1}^{N} \mathbf{R}_i^T \mathbf{D}_i \mathbf{A}_i^{-1} \mathbf{R}_i
\end{align*}
where $\left\{\mathbf{R}_i\right\}_{1 \le i \le N}$ are the Boolean restriction matrices from the global to the local finite element spaces and $\left\{\mathbf{D}_i\right\}_{1 \le i \le N}$ are local diagonal matrices representing the partition of unity. The key ingredient of the ORAS method is that the local subdomain matrices $\left\{\mathbf{A}_i\right\}_{1 \le i \le N}$ incorporate more efficient Robin transmission conditions.

Note that, unlike in \cite{Graham:2020:DDI} where the emphasis is placed on the independence of the one-level method to the wave number, we focus here on the scalability aspect, i.e., the independence of the one-level method with respect to the number of subdomains $N$ as soon as the absorption parameter $k\sigma$ is positive. We will observe that, beyond a sufficiently large value of $N$, the iteration count does not increase further, though in general this value will depend on the parameters of the problem, namely the wave number and absorption as well as the overlap and subdomain size. As a side effect, when the absorption is sufficiently large, i.e., of order $k$, wave number independence is also achieved.

In \Cref{Table:Itcount1} we detail the GMRES iteration count for an increasing number of subdomains $N$ and different values of $k$ for the wave-guide problem and the wave propagation in free space problem. We set the conductivity parameter as $\sigma = 1$ (giving an absorption parameter $k$). We see that, after an initial increase, the iteration counts become independent of the number of subdomains and also independent of the wave number, which is consistent with the results obtained in \cite{Graham:2020:DDI} where the absorption parameter for optimal convergence is of order $k$. Another possible explanation of this is that when the absorption parameter increases, the waves are damped and their amplitude will decrease with the distance to the boundary on which the excitation is imposed. Hence, when additional subdomains are added, the solution will not vary much in these subdomains and thus the residual will already be small.

\begin{table}[t]
	\centering
	\caption{Preconditioned GMRES iteration counts for varying wave number $k$ and number of subdomains $N$ when $\sigma = 1$.}
	\label{Table:Itcount1}
	\small
	\vspace*{-1mm}
	\begin{tabular}{c|cccccccc|cccccccc}
		& \multicolumn{8}{c|}{Wave-guide problem} & \multicolumn{8}{c}{Free space problem} \\
		\hline
		$k \backslash N$ & 8 & 16 & 24 & 32 & 40 & 48 & 56 & 64 & 8 & 16 & 24 & 32 & 40 & 48 & 56 & 64 \\
		\hline
		20  & 19 & 22 & 25 & 30 & 30 & 30 & 30 & 30 & 19 & 21 & 25 & 25 & 25 & 25 & 25 & 25 \\
		40  & 18 & 21 & 24 & 29 & 29 & 29 & 29 & 29 & 17 & 19 & 24 & 25 & 25 & 25 & 25 & 25 \\
		60  & 19 & 21 & 24 & 29 & 29 & 29 & 29 & 29 & 16 & 19 & 24 & 25 & 25 & 25 & 25 & 25 \\
		80  & 19 & 21 & 24 & 28 & 28 & 28 & 28 & 28 & 16 & 18 & 24 & 25 & 25 & 25 & 25 & 25 \\
		100 & 19 & 21 & 24 & 28 & 28 & 28 & 28 & 28 & 16 & 18 & 24 & 25 & 25 & 25 & 25 & 24
	\end{tabular}
	\vspace*{-2mm}
\end{table}

\section{Conclusions}
\label{sec:conc}

In this work we have analysed a purely iterative version of the Schwarz domain decomposition algorithm, in the limiting case of many subdomains, at the continuous level for the one-dimensional and two-dimensional Helmholtz and Maxwell's equations with absorption. The key mathematical tool which facilitated this study is the limiting spectrum of a sequence of block Toeplitz matrices having a particular structure, for which we proved a new result in the non-Hermitian case. The algorithm is convergent in the one-dimensional case as soon as we have absorption and, for sufficiently many subdomains $N$, its convergence factor becomes independent of the number of subdomains, meaning the algorithm is also scalable. In practice, this is achieved for relatively small $N$. In the two-dimensional case these conclusions remain true for the evanescent modes of the error (i.e., $\tilde k > k$) or when, between them, $\sigma$, $\delta$ and $L$ are sufficiently large. In particular, we proved that the stationary iteration will always converge when $\sigma \ge k$, giving an absorption parameter $k^2$. We further showed that the algorithm can be $k$-robust within certain scenarios, requiring the domain decomposition parameters to depend on $k$.

The concept of the limiting spectrum proved to be a very elegant mathematical tool and can be used, for example, in constructing more sophisticated transmission conditions, to further explore parameter robustness, to analyse the algorithm at the discrete level, or to design improved preconditioners.
\vspace*{4mm}

\bibliographystyle{siam}
\bibliography{paper_refs}

\end{document}